\documentclass[12pt,a4paper]{article}
\usepackage{amssymb}
\usepackage{amsmath}
\usepackage{latexsym}
\usepackage{amsthm}

\def\E{\mathbb E}

\newtheorem{lem}{Lemma}
\newtheorem{thm}{Theorem}
\newtheorem{thmletter}{Theorem}

\newtheorem{prop}[thm]{Proposition}
\theoremstyle{definition}
\newtheorem*{remark}{Remark}

\usepackage{xpatch}
\xpatchcmd{\proof}{\itshape}{\normalfont\proofnameformat}{}{}
\newcommand{\proofnameformat}{}

\begin{document}

\renewcommand{\proofnameformat}{\bfseries}

\begin{center}
{\large\textbf{Equidistribution of random walks on compact groups II. The Wasserstein metric}}

\vspace{5mm}

\textbf{Bence Borda}

{\footnotesize Graz University of Technology

Steyrergasse 30, 8010 Graz, Austria

Email: \texttt{borda@math.tugraz.at}}

\vspace{5mm}


\end{center}

\vspace{5mm}

\begin{abstract}
We consider a random walk $S_k$ with i.i.d.\ steps on a compact group equipped with a bi-invariant metric. We prove quantitative ergodic theorems for the sum $\sum_{k=1}^N f(S_k)$ with H\"older continuous test functions $f$, including the central limit theorem, the law of the iterated logarithm and an almost sure approximation by a Wiener process, provided the distribution of $S_k$ converges to the Haar measure in the $p$-Wasserstein metric fast enough. As an example we construct discrete random walks on an irrational lattice on the torus $\mathbb{R}^d/\mathbb{Z}^d$, and find their precise rate of convergence to uniformity in the $p$-Wasserstein metric. The proof uses a new Berry--Esseen type inequality for the $p$-Wasserstein metric on the torus, and the simultaneous Diophantine approximation properties of the lattice. These results complement the first part of this paper on random walks with an absolutely continuous component and quantitative ergodic theorems for Borel measurable test functions.
\end{abstract}

\section{Introduction}

Consider a compact Hausdorff group $G$ with normalized Haar measure $\mu$, and a random walk $S_k= \prod_{j=1}^k X_j =X_1 X_2 \cdots X_k$ on $G$, where $X_1, X_2, \dots$ are independent, identically distributed (i.i.d.) $G$-valued random variables. Assuming that the distribution of $X_1$ is a regular Borel probability measure $\nu$, the distribution of $S_k$ is $\nu^{*k}$, the $k$-fold convolution of $\nu$. We say that $\nu$ is \textit{adapted} if its support is not contained in any proper closed subgroup of $G$; further, $\nu$ is \textit{strictly aperiodic} if its support is not contained in a coset of any proper closed normal subgroup of $G$.

The study of such random walks is a classical topic; in the first part of this paper \cite{BO} we gave a brief overview of the early history. The most fundamental result is
\begin{thmletter}[Stromberg \cite{ST}]\label{Stromberg}
Let $\nu$ be a regular Borel probability measure on a compact Hausdorff group $G$. Then $\nu^{*k}$ converges weakly to the Haar measure $\mu$ as $k \to \infty$ if and only if $\nu$ is adapted and strictly aperiodic.
\end{thmletter}

Throughout the paper we shall work under the technical assumption that $G$ is second countable. This will ensure that $G$ is metrizable; in fact, every second countable compact Hausdorff group is metrizable by a bi-invariant metric $d$, that is, by a metric $d$ satisfying $d(ax,ay)=d(xa,ya)=d(x,y)$ for all $x,y,a \in G$ \cite[p.\ 71]{HR}. Recall that the Haar measure $\mu$ is also bi-invariant; that is, $\mu(aB)=\mu(Ba)=\mu(B)$ for all Borel sets $B \subseteq G$ and $a \in G$. Second countability also implies that every Borel probability measure on $G$ is regular.

The rate of convergence of $\nu^{*k}$ to the Haar measure $\mu$ in various probability metrics has been extensively studied, see Diaconis \cite{DI}. In contrast, here we consider the equidistribution properties of the random walk $S_k$; that is, we study sums of the form $\sum_{k=1}^N f(S_k)$ with various test functions $f: G \to \mathbb{R}$. Formally speaking, we consider the empirical distribution of the points $S_1, S_2, \dots, S_N$. This problem is often more difficult, as it is related to the joint distribution of the sequence $(S_1, S_2, \dots)$ in $G^{\mathbb{N}}$ instead of the marginals $\nu^{*k}$. Our starting point is
\begin{thmletter}[Berger--Evans \cite{BE}]\label{Berger} Let $G$ be a second countable compact Hausdorff group. Let $X_1, X_2, \dots$ be i.i.d.\ $G$-valued random variables, and set $S_k=\prod_{j=1}^k X_j$. Then
\begin{equation}\label{equidistribution}
\frac{1}{N} \sum_{k=1}^N f(S_k) \to \int_G f \, \mathrm{d}\mu \qquad \mathrm{a.s.}
\end{equation}
for all continuous functions $f:G \to \mathbb{R}$ if and only if the distribution of $X_1$ is adapted.
\end{thmletter}
Here a.s.\ (almost surely) means that the given relation holds with probability $1$. Theorem \ref{Berger} is an ergodic theorem; more precisely, it gives a necessary and sufficient condition for the random sequence $S_k$, $k=1,2,\dots$ to be equidistributed in $G$ with probability $1$. Our main goal is to prove quantitative ergodic theorems for the random walk $S_k$. In the terminology of probability theory, this means that we aim to improve the strong law of large numbers \eqref{equidistribution} to more delicate limit theorems such as the central limit theorem (CLT) or the law of the iterated logarithm (LIL). To prove a quantitative version of Theorem \ref{Berger}, we will need stronger, more quantitative assumptions on the distribution of $X_1$ and on the test function $f$. Recall that the $p$-Wasserstein metric $W_p$ metrizes weak convergence of Borel probability measures on $G$ for any $p>0$; in particular by Theorem \ref{Stromberg}, $W_p (\nu^{*k}, \mu) \to 0$ as $k \to \infty$ if and only if $\nu$ is adapted and strictly aperiodic. In our quantitative form of Theorem \ref{Berger} instead of the adaptedness of $\nu$, we shall assume that $W_p (\nu^{*k}, \mu) \to 0$ fast enough. Regarding the test function $f$, instead of continuity, we shall assume that it is $p$-H\"older for some $0<p \le 1$; recall that a function $f:G \to \mathbb{R}$ is called $p$-H\"older if there exists a constant $K \ge 0$ such that $|f(x)-f(y)| \le K d(x,y)^p$ for all $x,y \in G$, where $d$ is the metric on $G$. Our quantitative ergodic theorem is
\begin{thm}\label{theoremCLTLIL} Let $G$ be a second countable compact Hausdorff group with normalized Haar measure $\mu$, and fix a bi-invariant metric on $G$ which metrizes its topology. Consider a sequence $X_1, X_2, \dots$ of i.i.d.\ $G$-valued random variables with distribution $\nu$, and set $S_k=\prod_{j=1}^k X_j$. Let $0<p\le 1$, and assume that $\sum_{k=1}^{\infty} W_p (\nu^{*k}, \mu) < \infty$. Then for any $p$-H\"older function $f:G \to \mathbb{R}$ such that $\int_G f \, \mathrm{d} \mu =0$, the sum $\sum_{k=1}^N f(S_k)$ satisfies the central limit theorem
\begin{equation}\label{mainCLT}
\frac{\sum_{k=1}^N f(S_k)}{\sqrt{N}} \overset{d}{\to} \mathcal{N}(0,\sigma^2 )
\end{equation}
and the law of the iterated logarithm
\begin{equation}\label{mainLIL}
\limsup_{N \to \infty}\frac{\sum_{k=1}^N f(S_k)}{\sqrt{2 N \log \log N}} = \sigma \qquad \mathrm{a.s.}
\end{equation}
with some constant $0 \le \sigma <\infty$ depending only on $f$ and $\nu$.
\end{thm}
Relation \eqref{mainCLT} means convergence in distribution to the normal distribution with mean zero and variance $\sigma^2$ (or to the constant $0$ in case $\sigma =0$). We emphasize that the $p$-Wasserstein metric and $p$-H\"older functions are defined with respect to the same bi-invariant metric. We conjecture that the condition $\sum_{k=1}^{\infty} W_p(\nu^{*k}, \mu) < \infty$ is best possible. Theorem \ref{theoremCLTLIL} is in fact a direct corollary of our main result, an almost sure approximation of the sum $\sum_{k=1}^N f(S_k)$ by a Wiener process, see Theorem \ref{theoremWIENER}.

In the first part of this paper \cite{BO} we considered the class of all Borel measurable test functions. As an analogue of Theorem \ref{Berger} we proved that the strong law of large numbers \eqref{equidistribution} holds for all bounded Borel measurable functions $f$ if and only if the distribution of $X_1$ is adapted, and the distribution of $S_k$ has an absolutely continuous component for some $k \ge 1$. We also proved quantitative forms of equidistribution such as the CLT and the LIL, and showed that the assumption that $f$ is bounded can be weakened to delicate moment conditions. The main message of our first paper \cite{BO} is that random walks with an absolutely continuous component satisfy quantitative ergodic theorems for the much wider class of Borel measurable test functions. In the present paper we will thus focus on random walks with singular distributions; see Theorems \ref{rankrlattice1} and \ref{rankrlattice2} for examples of discrete distributions to which Theorem \ref{theoremCLTLIL} applies.

Quantitative forms of Kakutani's random ergodic theorem for ergodic group actions along random walks were given by Furman and Shalom \cite{FS}. Their results apply in particular to the action of certain subgroups of $\mathrm{SL}_d(\mathbb{Z})$ on the torus $\mathbb{R}^d/\mathbb{Z}^d$. Such actions along random walks have more recently been studied by Bourgain, Furman, Lindenstrauss and Mozes \cite{BFLM} and by Boyer \cite{BOY}. For random walks on $\mathrm{GL}_d(\mathbb{R})$ see Cuny, Dedecker and Jan \cite{CDJ} and references therein. The case of Abelian groups is of course simpler; for instance, Theorem \ref{Berger} for the circle group $G=\mathbb{R}/\mathbb{Z}$ dates back to Robbins in 1953 \cite{RO}. Berkes and Borda \cite{BB3} proved quantitative ergodic theorems on the circle group for test functions of bounded variation. The same authors also showed that the LIL \eqref{mainLIL} with $f(x)=\exp (2 \pi i x)$ on the circle group holds if and only if the distribution of $X_1$ is nondegenerate \cite{BB2}. It is natural to ask whether this fact generalizes to arbitrary compact groups; for example, one might conjecture that the CLT \eqref{mainCLT} and the LIL \eqref{mainLIL} with some $0 \le \sigma < \infty$ remain true for characters $f(x)=\mathrm{tr}\, \pi (x)$, $\pi$ a nontrivial irreducible unitary representation of $G$, under the sole assumption that $\nu$ is adapted (without assuming fast enough convergence of $\nu^{*k}$ in some probability metric).

The proof of Theorem \ref{theoremCLTLIL} (and our main result Theorem \ref{theoremWIENER}) is based on a perturbation method. Simply put, in the long product $S_k=X_1 X_2 \cdots X_k$ we will replace the product $S_J:=\prod_{j \in J} X_j$ over a short block $J \subset [1,k]$ by a uniformly distributed variable $U_J$. Surprisingly, this perturbation will introduce \textit{independence}, and limit theorems such as the CLT and the LIL follow from classical methods of probability theory. In order for the error of this perturbation to be small, we need to use a uniformly distributed variable $U_J$ which is, in a sense, close to $S_J$; in other words, we need a suitable coupling of $\nu^{*|J|}$ and the Haar measure $\mu$. This method is quite flexible, and works whenever a metric $\delta$ on the set of Borel probability measures on $G$ is defined as
\begin{equation}\label{deltadefinition}
\delta (\nu_1, \nu_2) = \sup_{f \in \mathcal{F}} \left| \int_G f \, \mathrm{d} \nu_1 - \int_G f \, \mathrm{d} \nu_2 \right|
\end{equation}
with a suitable class of test functions $\mathcal{F}$, and $\delta (\nu^{*k} , \mu) < \varepsilon$ implies the existence of a suitable coupling of $\nu^{*k}$ and the Haar measure $\mu$. In this case the sum $\sum_{k=1}^N f(S_k)$ satisfies limit theorems such as the CLT and the LIL for all $f \in \mathcal{F}$, provided $\delta (\nu^{*k},\mu) \to 0$ fast enough.

We find the perturbation method outlined above remarkable for three reasons: it is simple, flexible and completely Fourier analysis free. In its simplest form it was first used by Schatte \cite{SCH1}, \cite{SCH2} with the uniform metric (also called Kolmogorov metric) $\delta_{\mathrm{unif}}(\nu_1, \nu_2):=\sup_{x \in [0,1]} |\nu_1 ([0,x)) - \nu_2 ([0,x))|$ on the set of Borel probability measures on the circle group $G=\mathbb{R}/\mathbb{Z}$, identified with the interval $[0,1)$. Note that the uniform metric is also of the form \eqref{deltadefinition} with the class of functions of bounded variation; further, if $\delta_{\mathrm{unif}}(\nu_1, \mu) < \varepsilon$, then there exists a coupling $(X,U)$ of $\nu_1$ and $\mu$ with $|X-U|<\varepsilon$ a.s. For the details see Berkes and Borda \cite{BB3}.

What we do here is a far-reaching generalization of Schatte's perturbation method: we show that it works on arbitrary compact groups with various probability metrics. In the first part of this paper \cite{BO} we used this method with the total variation metric $\| \cdot \|_{\mathrm{TV}}$ and the class of bounded Borel measurable test functions; recall that $\| \nu_1 - \nu_2 \|_{\mathrm{TV}}<\varepsilon$ implies the existence of a coupling $(X,Y)$ of $\nu_1$ and $\nu_2$ with $\Pr (X \neq Y)<\varepsilon$. In the present paper we work out the details of this approach for the $p$-Wasserstein metric $W_p$ and the class of $p$-H\"older test functions; recall that $W_p(\nu_1, \nu_2) < \varepsilon$ implies the existence of a coupling $(X,Y)$ of $\nu_1$ and $\nu_2$ with $\E \, d(X,Y)^p < \varepsilon$.

Theorem \ref{theoremCLTLIL} (and our main result Theorem \ref{theoremWIENER}) raise two natural questions: first, how to estimate the rate of convergence of $\nu^{*k}$ to the Haar measure $\mu$ in the $p$-Wasserstein metric to tell whether our results apply to a given random walk; second, how to characterize the test functions $f$ for which $\sigma =0$. While the proof of Theorems \ref{theoremCLTLIL} and \ref{theoremWIENER} are completely Fourier analysis free, to settle these questions Fourier methods seem inevitable. In this paper we only consider these two questions in the simplest case of the additive group of the $d$-dimensional torus $G=\mathbb{R}^d/\mathbb{Z}^d$. For the sake of simplicity, we use the Euclidean metric on $\mathbb{R}^d/\mathbb{Z}^d$ and the maximum metric $\| \cdot \|_{\infty}$ on its unitary dual $\mathbb{Z}^d$. Further, $\widehat{\nu}$ denotes the Fourier transform of a Borel probability measure $\nu$. The second question has a very simple answer on the torus.
\begin{prop}\label{sigmaequalzero} In the case of the torus $G=\mathbb{R}^d/\mathbb{Z}^d$ with the Euclidean metric, in Theorem \ref{theoremCLTLIL} we have $\sigma =0$ if and only if $f=0$.
\end{prop}

To estimate the rate of convergence in the $p$-Wasserstein metric, we will use a Berry--Esseen type inequality. The corresponding result for the uniform metric on $\mathbb{R}^d/\mathbb{Z}^d$ is due to Niederreiter and Philipp \cite{NP}. The original Berry--Esseen inequality concerns the uniform metric on $\mathbb{R}$ \cite[p.\ 142]{P}; a generalization for the uniform metric on $\mathbb{R}^d$ is due to von Bahr \cite{vB}. Bobkov \cite{BOB3} gives a survey of similar smoothing inequalities for several other probability metrics on $\mathbb{R}$.
\begin{prop}[Berry--Esseen inequality for $W_1$ on the torus]\label{ErdosTuran} Let $\nu_1$ and $\nu_2$ be Borel probability measures on $\mathbb{R}^d/\mathbb{Z}^d$ with the Euclidean metric. For any integer $H \ge 1$,
\[ W_1 (\nu_1, \nu_2 ) \le \frac{6d}{H} + \frac{d^{1/2}}{2 \pi} \left( \sum_{\substack{h \in \mathbb{Z}^d \\ 0< \| h \|_{\infty}<H}} \frac{|\widehat{\nu_1}(h) - \widehat{\nu_2}(h)|^2}{|h|^2} \right)^{1/2} . \]
\end{prop}
The proof is based on smoothing with the normalized square of the $d$-dimensional Fej\'er kernel as an approximate identity. An application of the H\"older inequality shows that $W_p(\nu_1, \nu_2) \le W_1 (\nu_1, \nu_2)^p$ for all $0<p \le 1$, thus Proposition \ref{ErdosTuran} also yields an upper bound for $W_p$, $0<p \le 1$; alternatively, a straightforward modification of the proof gives a Berry--Esseen type inequality directly for $W_p$, $0<p\le 1$. We mention that a similar result has already been available in the case $d=1$. Indeed, in one dimension every Lipschitz function is of bounded variation, and so Koksma's inequality \cite[p.\ 143]{KN} combined with the Berry--Esseen type inequality of Niederreiter and Philipp for the uniform metric yields an upper estimate for $W_1$. This deduction no longer works in higher dimensions; the reason is that for every $d \ge 2$ there exist Lipschitz functions on $\mathbb{R}^d/\mathbb{Z}^d$ which are not of bounded variation in the sense of Hardy and Krause (see \cite[Section 3.1]{BA} for a construction). Consequently, for $d \ge 2$ Proposition \ref{ErdosTuran} and the theorem of Niederreiter and Philipp are not comparable.

We emphasize that neither the statement nor the proof of Proposition \ref{ErdosTuran} uses concepts such as axis parallel boxes or functions of bounded variation; the only ingredients are the metric on $\mathbb{R}^d/\mathbb{Z}^d$ and the unitary dual. Generalizing this approach, smoothing inequalities can be proved in other classical compact groups as well, yielding examples of random walks on noncommutative compact groups to which Theorem \ref{theoremCLTLIL} applies. The details will be given elsewhere.

With Proposition \ref{ErdosTuran} at our disposal, it is not difficult to find examples of random walks with fast convergence in the $p$-Wasserstein metric. We work out the details for certain random walks on rank $r$ lattices. We mention a simple explicit construction here; for more general results see Section \ref{speedsection}. Let $r,d \ge 1$ be integers, and consider a monic irreducible polynomial with integer coefficients of degree $r+d$, with distinct real roots $a_1, a_2, \dots, a_{r+d}$. In other words, $a_1, a_2, \dots, a_{r+d}$ are conjugate real algebraic integers. Let $V$ be the $r \times r$ matrix with entries $V_{ij}=a_i^{j-1}$, and let $W$ be the $r \times d$ matrix with entries $W_{ij}=a_i^{r+j-1}$. Note that $V$ is an invertible Vandermonde matrix. Consider the $r \times d$ matrix $M=V^{-1}W$.
\begin{thm}\label{rankrlattice1} Let $\alpha_1, \alpha_2, \dots, \alpha_r \in \mathbb{R}^d$ denote the $r$ rows of the matrix $M$, and let $\nu$ be the Borel probability measure on $\mathbb{R}^d/\mathbb{Z}^d$ with atoms $\pm \alpha_i + \mathbb{Z}^d$, $1 \le i \le r$ of weight $1/(2r)$ each. For any $0<p \le 1$,
\[ k^{-r/(2d)} \ll W_p(\nu^{*k},\mu)^{1/p} \ll k^{-r/(2d)} \]
with implied constants depending only on the matrix $M$. In particular, if $pr>2d$, then $\sum_{k=1}^{\infty} W_p (\nu^{*k},\mu)<\infty$, and Theorem \ref{theoremCLTLIL} applies.
\end{thm}
The entries of $M$ are explicitly computable rational functions of the algebraic integers $a_1, a_2, \dots, a_r$. The conjugates $a_{r+1}, a_{r+2}, \dots, a_{r+d}$ only appear in the proof of the fact that the system of vectors $\alpha_1, \alpha_2, \dots, \alpha_r$ is badly approximable; see Section \ref{speedsection}.

Theorem \ref{rankrlattice1} should be compared with the rate of convergence in the $p$-Wasserstein metric in the classical CLT setting. Given $p \ge 1$ and i.i.d.\ real-valued random variables $\zeta_1, \zeta_2, \dots$ with mean zero, unit variance and $\E |\zeta_1|^{p+2}< \infty$, the distance of $k^{-1/2} \sum_{j=1}^k \zeta_j$ from the standard normal distribution in $W_p$ is $O(k^{-1/2})$, and this is best possible. For the proof and more general results see Bobkov \cite{BOB1} and references therein.

The notation used in the paper and definitions are given in Section \ref{notation}. The main result about random walks on general compact groups is stated in Section \ref{pholder}, and proved in Section \ref{proofsection}. The main results about the rate of convergence of random walks on the torus are stated in Section \ref{speedsection}; their proof, together with the proof of Propositions \ref{sigmaequalzero} and \ref{ErdosTuran} are given in Section \ref{torussection}.

\section{Main results}

\subsection{Definitions and notation}\label{notation}

For the rest of the paper $G$ denotes a compact group whose topology is generated by a bi-invariant metric $d$, with normalized Haar measure $\mu$. We write $\nu_1 * \nu_2$ for the convolution of the Borel probability measures $\nu_1$ and $\nu_2$ on $G$, and $\nu^{*k}$ for the $k$-fold convolution of $\nu$. Recall that $\nu * \mu = \mu * \nu = \mu$ for any $\nu$.

Let $\mathcal{F}_p$ ($0<p \le 1$) denote the set of all $p$-H\"older functions $f:G \to \mathbb{R}$ such that $\int_G f \, \mathrm{d} \mu =0$. The optimal $p$-H\"older constant of $f$ will be denoted by
\[ \| f \|_{p-\textrm{H\"old}} := \sup_{\substack{x,y \in G \\ x \neq y}} \frac{|f(x)-f(y)|}{d(x,y)^p} . \]
Observe that $\| \cdot \|_{p-\textrm{H\"old}}$ is a norm on $\mathcal{F}_p$. We simply write $\| f \|_q$ for the $L^q(G,\mu)$-norm of $f$ ($1 \le q \le \infty$).

The $p$-Wasserstein distance ($0<p \le 1$) of two Borel probability measures $\nu_1$ and $\nu_2$ on $G$ is defined as\footnote{The exponent $1/p$ is missing to ensure that $W_p$ is a metric, i.e.\ it satisfies the triangle inequality.}
\[ W_p (\nu_1, \nu_2 ) := \inf_{\vartheta \in \mathrm{Coup} (\nu_1, \nu_2 )} \int_{G \times G} d(x,y)^p \, \mathrm{d}\vartheta (x,y) , \]
where $\mathrm{Coup}(\nu_1, \nu_2)$ denotes the set of couplings of $\nu_1$ and $\nu_2$, i.e.\ the set of Borel probability measures $\vartheta$ on $G \times G$ with marginals $\vartheta (B \times G)=\nu_1 (B)$ and $\vartheta (G \times B) = \nu_2 (B)$ ($B \subseteq G$ Borel). By the Kantorovich duality theorem,
\[ W_p (\nu_1, \nu_2 ) = \sup_{\substack{f \in \mathcal{F}_p \\ \| f \|_{p-\textrm{H\"old}} \le 1}} \left| \int_G f \, \mathrm{d} \nu_1 - \int_G f \, \mathrm{d} \nu_2 \right| . \]
This duality is usually stated for $p=1$, i.e.\ for Lipschitz functions; to see the general case simply note that $d(x,y)^p$ is also a metric on $G$ generating the same topology, and apply Kantorovich duality for Lipschitz functions in the $d(x,y)^p$ metric. Observe that for any Borel probability measures $\nu_1$, $\nu_2$ and $\nu$ on $G$ we have $W_p(\nu_1 * \nu, \nu_2 * \nu) \le W_p (\nu_1, \nu_2)$.

A $G$-valued random variable is a Borel measurable map $X$ from a probability space to $G$; the distribution of $X$ is the Borel probability measure on $G$ which assigns $\Pr (X \in B)$ to a Borel set $B \subseteq G$. A $G$-valued random variable is called uniformly distributed if its distribution is the Haar measure $\mu$. If the $G$-valued random variables $X$ and $Y$ are independent, then the distribution of $XY$ is the convolution of the distributions of $X$ and $Y$.

Throughout $| \cdot |$ denotes the Euclidean norm of a vector in $\mathbb{R}^d$ or the cardinality of a set. We use $\langle x,y \rangle=\sum_{i=1}^d x_i y_i$ for the usual scalar product on $\mathbb{R}^d$. The $d$-dimensional torus $\mathbb{R}^d / \mathbb{Z}^d$ is identified by $[0,1)^d$ or $[-1/2,1/2)^d$ the usual way, as are functions on $\mathbb{R}^d/\mathbb{Z}^d$ by $\mathbb{Z}^d$-periodic functions on $\mathbb{R}^d$. The normalized Haar measure $\mu$ becomes the Lebesgue measure on a unit cube via this identification; we simply write $\mathrm{d}x$ in integrals with respect to the Lebesgue measure. Given $x,y \in \mathbb{R}^d$, we define the Euclidean distance from the point $x+\mathbb{Z}^d$ to the point $y+\mathbb{Z}^d$ on the torus as $\| x-y \|_{\mathbb{R}^d/\mathbb{Z}^d}:=\min_{m \in \mathbb{Z}^d} |x-y-m|$. In particular, $\| \cdot \|_{\mathbb{R}/\mathbb{Z}}$ is the distance from the nearest integer function. We use the maximum norm $\| h \|_{\infty}=\max_{1 \le j \le d} |h_j|$ for lattice points $h \in \mathbb{Z}^d$. The Fourier coefficients of a function $f: \mathbb{R}^d / \mathbb{Z}^d \to \mathbb{R}$ are denoted by $\widehat{f}(h)=\int_{[0,1)^d} e^{-2 \pi i \langle h,x \rangle} f(x) \, \mathrm{d}x$, $h \in \mathbb{Z}^d$; those of a Borel probability measure $\nu$ by $\widehat{\nu}(h)=\int_{[0,1)^d} e^{-2 \pi i \langle h,x \rangle} \, \mathrm{d}\nu (x)$, $h \in \mathbb{Z}^d$.

Given a (deterministic) function $E(t)$, $t\in [0,\infty)$, we define two stochastic processes $Y(t)$ and $Z(t)$ in the Skorokhod space $D[0,\infty)$ to be $o(E(t))$-equivalent the same way as in \cite{BO}. Somewhat informally, this notion simply means that $Y(t)=Z(t)+o(E(t))$ as $t \to \infty$ a.s., except we allow a process to be replaced by another process defined on a new probability space with the same distribution in $D[0,\infty)$.

For general facts about topological groups we refer to Hewitt and Ross \cite{HR}. The $p$-Wasserstein metric and Kantorovich duality are related to the theory of optimal transportation, see Villani \cite{V}. Simultaneous Diophantine approximation is covered in detail in Cassels \cite{CA}, Schmidt \cite{WSCH} and Sprindzuk \cite{SPR}.

\subsection{$p$-H\"older test functions on a compact group}\label{pholder}

Our main result is an almost sure approximation of the sum $\sum_{k=1}^N f(S_k)$ by a Wiener process. For any $f \in \mathcal{F}_p$ let
\begin{equation}\label{Cfnu}
C(f,\nu ) = \E f(U)^2 + 2 \sum_{k=1}^{\infty} \E f(U)f(US_k),
\end{equation}
where $U$ is a uniformly distributed $G$-valued random variable, independent of $X_1, X_2, \dots$.

\begin{thm}\label{theoremWIENER} Assume that the conditions of Theorem \ref{theoremCLTLIL} hold, and let $f \in \mathcal{F}_p$. Then the series in \eqref{Cfnu} is absolutely convergent, and $C(f,\nu) \ge 0$.
\begin{enumerate}
\item[(i)] If $C(f, \nu)>0$, then there exists a (deterministic) nondecreasing positive function $\sigma (t)$, $t \in [0,\infty)$ with $\sigma (t)^2=C(f, \nu) t +o(t)$ such that the processes $\sum_{k \le t} f(S_k)$ and $W(\sigma(t)^2)$ are $o\left( t^{5/12+\varepsilon} \right)$-equivalent in the Skorokhod space $D[0,\infty)$ for any $\varepsilon >0$, where $W(t)$, $t\in [0,\infty )$ is a standard Wiener process.
\item[(ii)] If $C(f, \nu)=0$, then
\[ \frac{\sum_{k \le t} f(S_k)}{\sqrt{t}} \overset{d}{\to} 0 \qquad \textrm{and} \qquad \frac{\sum_{k \le t} f(S_k)}{\sqrt{t \log \log t}} \to 0 \quad \mathrm{a.s.} \]
\end{enumerate}
\end{thm}

\begin{remark} As the proof will clearly show, the error term in $\sigma (t)^2$ is related to the tails of the series $\sum_{k=1}^{\infty} W_p (\nu^{*k}, \mu)$. For instance, if $W_p (\nu^{*k}, \mu) \ll k^{-(1+c)}$ for some $0<c \le 1/2$, then $\sigma (t)^2 = C(f, \nu ) t +O(t^{1-c/3})$. In this case $\sum_{k \le t} f(S_k)$ is $o(t^{1/2-c/6+\varepsilon})$-equivalent to $\sqrt{C(f,\nu)} W(t)$ for any $\varepsilon >0$.
\end{remark}

It was not our intention to find the optimal error term in Theorem \ref{theoremWIENER}; indeed, it is likely that the exponent $5/12$ can be improved. The point is that $5/12 < 1/2$, and so Theorem \ref{theoremCLTLIL} with $\sigma = \sqrt{C(f,\nu)}$ follows immediately from Theorem \ref{theoremWIENER} and classical limit theorems for Wiener processes. Furthermore, the almost sure asymptotics and the limit distribution of continuous functionals of the process $\sum_{k \le t} f(S_k)$ also follow. Of course we could also use the piecewise linear functions $\sum_{k \le \lfloor t \rfloor} f(S_k)+ (t-\lfloor t \rfloor ) f(S_{\lfloor t \rfloor +1})$ instead of the step functions $\sum_{k \le t} f(S_k)$; in that case the same results hold in the space of continuous functions $C[0,\infty)$.

The normalizing factor $C(f, \nu)$ defined in \eqref{Cfnu} is the same as in the first part of this paper \cite{BO}. For a brief comparison with general Markov chain theory we refer to the same paper; in particular, note that the same factor $C(f, \nu)$ appears in the variance in the classical CLT and LIL for Markov chains.

\subsection{Rate of convergence in the $p$-Wasserstein metric}\label{speedsection}

Let $d \ge 1$ and consider the additive group of the $d$-dimensional torus $\mathbb{R}^d/\mathbb{Z}^d$ with the Euclidean metric and normalized Haar measure $\mu$. Let $r \ge 1$ and $\alpha_1, \alpha_2, \dots, \alpha_r \in \mathbb{R}^d$. Further, let $\xi_1, \xi_2, \dots, \xi_r$ be $\mathbb{Z}$-valued random variables, and let $I$ be a $\{ 1,2,\dots, r \}$-valued random variable with $\xi_1, \xi_2, \dots, \xi_r, I$ independent. Consider the i.i.d.\ sequence $X_1, X_2, \dots$ of $\mathbb{R}^d/\mathbb{Z}^d$-valued random variables where $X_1=\xi_I \alpha_I+\mathbb{Z}^d$; that is, the distribution of $X_1$ is
\[ \nu=\sum_{i=1}^r \sum_{n \in \mathbb{Z}} \Pr (I=i) \Pr (\xi_i =n) \delta_{n\alpha_i+\mathbb{Z}^d}, \]
where $\delta_{a+\mathbb{Z}^d}$ is the Dirac measure concentrated on the point $a+\mathbb{Z}^d$ on the torus. We thus have a random walk on the rank $r$ lattice
\[ \left\{ \sum_{i=1}^r n_i \alpha_i + \mathbb{Z}^d \,\, : \,\, n_1, n_2, \dots, n_r \in \mathbb{Z} \right\} \subset \mathbb{R}^d / \mathbb{Z}^d, \]
and to generate an elementary step of the walk we first randomly choose one of the basis vectors $\alpha_i$ of the lattice, and then move forward by a random integral multiple of $\alpha_i$. The random walk in Theorem \ref{rankrlattice1} is a special case, when $I$ is uniformly distributed on $\{ 1,2,\dots, r \}$, and $\xi_1, \xi_2, \dots, \xi_r$ are Bernoulli random variables taking the values $\pm 1$ with probability $1/2$ each.

It is not surprising that the rate of convergence of this random walk to uniformity is sensitive to the Diophantine properties of the given lattice. We will work under the assumption
\begin{equation}\label{diophantinecond}
\inf_{\substack{h \in \mathbb{Z}^d \\ h \neq 0}} \left( \max_{1 \le i \le r} \left\| \langle h, \alpha_i \rangle \right\|_{\mathbb{R}/\mathbb{Z}} \cdot \psi (\| h \|_{\infty}) \right) >0,
\end{equation}
where $\psi (x)$ is a nondecreasing positive function on $[1,\infty)$.

Dirichlet's theorem \cite[p.\ 28]{WSCH} states that \eqref{diophantinecond} cannot hold with a function $\psi(x)=o(x^{d/r})$. A system of vectors $\alpha_1, \alpha_2, \dots, \alpha_r$ is called \textit{badly approximable} if \eqref{diophantinecond} is satisfied with $\psi (x)=x^{d/r}$. It is known that badly approximable systems of vectors exist for any $d,r \ge 1$; in fact, the system used in Theorem \ref{rankrlattice1} is badly approximable \cite[p.\ 43--45]{WSCH}. A system of vectors is called a \textit{Roth system} if every coordinate of every vector is algebraic, and \eqref{diophantinecond} holds with $\psi (x)=x^{d/r+\varepsilon}$ for any $\varepsilon>0$. Schmidt \cite[p.\ 157]{WSCH} gave a simple necessary and sufficient condition for being a Roth system in terms of the rank of certain linear forms on rational subspaces. For instance, in the special case $d=1$ a system $\alpha_1, \alpha_2, \dots, \alpha_r$ is a Roth system if and only if $1,\alpha_1, \alpha_2, \dots, \alpha_r$ are algebraic and linearly independent over the rationals. The corresponding metric result is the Khintchine--Groshev theorem \cite{SPR}. Assume $\psi(x)=x^{d/r} L(x)^{1/r}$ with a nondecreasing positive function $L(x)$. If the series $\sum_{k=1}^{\infty}1/(k L(k))$ is convergent (resp.\ divergent), then \eqref{diophantinecond} is satisfied by almost every (resp.\ almost no) system $(\alpha_1, \alpha_2, \dots, \alpha_r) \in \mathbb{R}^{rd}$ in the sense of the Lebesgue measure.
\begin{thm}\label{rankrlattice2} Assume that $\xi_1, \xi_2, \dots, \xi_r$ are nondegenerate, $\Pr (I=i)>0$ for all $1 \le i \le r$, and \eqref{diophantinecond} is satisfied with a positive function $\psi (x)$ such that $\psi(x)x^{-d/r}$ is nondecreasing. For any integer $k \ge 1$,
\[ W_1 (\nu^{*k}, \mu ) \ll \frac{1}{\psi^{-1}(\sqrt{k})}, \]
where $\psi^{-1}(y)= \sup \{ x \ge 1 : \psi (x) \le y \}$ is the generalized inverse function of $\psi$. The implied constant depends only on the distributions of $\xi_1, \xi_2, \dots, \xi_r, I$, the value of the infimum in \eqref{diophantinecond}, $d$ and $r$.
\end{thm}
Theorem \ref{rankrlattice2} with $\psi (x)=x^{d/r}$ yields the estimate
\[ W_1 (\nu^{*k}, \mu ) \ll k^{-r/(2d)}, \]
valid whenever the system of vectors $\alpha_1, \alpha_2, \dots, \alpha_r$ is badly approximable; in particular, the upper bound in Theorem \ref{rankrlattice1} follows. If \eqref{diophantinecond} is satisfied with $\psi (x)=x^{d/r+\varepsilon}$ for any $\varepsilon >0$, then Theorem \ref{rankrlattice2} gives
\[ W_1 (\nu^{*k}, \mu ) \ll k^{-r/(2d)+\varepsilon} \]
for any $\varepsilon >0$. This estimate holds for any Roth system $\alpha_1, \alpha_2, \dots, \alpha_r$, and also for almost every $(\alpha_1, \alpha_2, \dots, \alpha_r) \in \mathbb{R}^{rd}$.

We conjecture that the estimate is tight whenever $\xi_1, \xi_2, \dots, \xi_r$ have finite variance, possibly under certain regularity assumptions on the distributions of $\xi_1, \xi_2, \dots, \xi_r, I$. We were only able to prove this conjecture for badly approximable systems.
\begin{thm}\label{rankrlattice3} Assume that $\Pr (I=i) \E \xi_i^2 \le B$ for all $1 \le i \le r$ with some constant $B>0$. If the system of vectors $\alpha_1, \alpha_2, \dots, \alpha_r \in \mathbb{R}^d$ is badly approximable, then for any $0<p \le 1$  and any integer $k \ge 1$,
\[ W_p (\nu^{*k}, \mu )^{1/p} \gg k^{-r/(2d)} \]
with an implied constant depending only on $\alpha_1, \alpha_2, \dots, \alpha_r$ and $B$.
\end{thm}

Diophantine approximation has been used by several authors to estimate the rate of convergence of random walks. Hensley and Su \cite{HS} showed that if $d=1$, the system $\alpha_1, \alpha_2, \dots, \alpha_r$ is badly approximable and $\nu$ has atoms $\pm \alpha_i + \mathbb{Z}$, $1 \le i \le r$ of weight $1/(2r)$ each, then the rate of convergence in the uniform metric is
\[ k^{-r/2} \ll \delta_{\mathrm{unif}}(\nu^{*k}, \mu) \ll k^{-r/2} . \]
Berkes and Borda \cite{BB1} studied the case $d=1$, $r=1$, and proved that if $\xi_1$ has finite variance and satisfies certain regularity assumptions (e.g.\ if $\xi_1$ is a Bernoulli variable), then
\[ k^{-1/(2 \gamma )} \ll \delta_{\mathrm{unif}} (\nu^{*k}, \mu ) \ll k^{-1/(2 \gamma)}, \]
provided $\alpha_1$ satisfies \eqref{diophantinecond} with $\psi (x)=x^{\gamma}$ with some $\gamma \ge 1$ in a tight way; however, the lower estimate only holds for infinitely many $k$ (not necessarily for all $k$). See the same paper for results about heavy-tailed $\xi_1$. Bobkov \cite{BOB2} considered random walks on the real line defined in terms of irrational numbers in a similar way, and proved tight estimates for their distance from the normal distribution in the uniform metric, as well as Edgeworth expansions.

\section{Random walks on general compact groups}\label{proofsection}

Throughout this section we assume that the conditions of Theorem \ref{theoremCLTLIL} hold. Let us fix $0<p \le 1$, and let $\Delta_k=W_p (\nu^{*k}, \mu)$ denote the distance of $S_k$ from the uniform distribution in the $p$-Wasserstein metric. Further, let $U$ be a uniformly distributed $G$-valued random variable independent of $X_1, X_2, \dots$.

First of all note that $\| f \|_{\infty} \ll \| f \|_{p-\textrm{H\"old}}$ for any $f \in \mathcal{F}_p$. Indeed, for any $x \in G$ we have
\[ |f(x)| = \left| \int_G \left( f(x)-f(y) \right) \, \mathrm{d} \mu (y) \right| \le \| f \|_{p-\textrm{H\"old}} \int_G d(x,y)^p \, \mathrm{d} \mu (y) , \]
where the last integral does not depend on $x$ by the translation invariance of $d$ and $\mu$, and is finite by the compactness of $G$. Observe also that
\[ W_p (\nu^{*(k+1)}, \mu ) = W_p \left( \nu^{*k} * \nu , \mu * \nu \right) \le W_p (\nu^{*k} , \mu ) , \]
therefore the sequence $\Delta_k$ is nonincreasing. The assumption $\sum_{k=1}^{\infty} \Delta_k < \infty$ thus implies that $k \Delta_k \to 0$ as $k \to \infty$. Our main tool is the fact that by Kantorovich duality, any $p$-H\"older function $g: G \to \mathbb{R}$ satisfies
\begin{equation}\label{maintool}
\left| \E g(S_k) - \E g(U) \right| = \left| \int_G g \, \mathrm{d} \nu^{*k} - \int_G g \, \mathrm{d} \mu \right| \le \| g \|_{p-\textrm{H\"old}} \Delta_k .
\end{equation}

\subsection{The variance}

We start by finding the precise asymptotics of the variance of the modified sum $\sum_{k=1}^N f(US_k)$, and estimate the variance of $\sum_{k=1}^N f(S_k)$. It would not be difficult to find the precise asymptotics of the latter sum as well; however, we will not need it later.

\begin{lem}\label{variancefusk} For any $f \in \mathcal{F}_p$ the series in \eqref{Cfnu} is absolutely convergent, and
\[ \E \left( \sum_{k=1}^N f(US_k) \right)^2 = C(f, \nu ) N + o(N) \qquad \mathrm{as} \,\,\, N \to \infty . \]
In particular, $C(f, \nu) \ge 0$.
\end{lem}

\begin{proof} Let $A_k = \E f(U)f(US_k)$. Since $U$ is independent of $S_k$, we have
\[ A_k = \int_G \int_G f(u)f(ux) \, \mathrm{d} \mu (u) \mathrm{d} \nu^{*k}(x). \]
Letting $g(x)=\int_G f(u)f(ux) \, \mathrm{d} \mu (u)$ denote the inner integral, we can thus write $A_k=\E g(S_k)$. Observe that $\int_G g \, \mathrm{d} \mu =0$, and $\| g \|_{p-\textrm{H\"old}} \le \| f \|_1 \cdot \| f \|_{p-\textrm{H\"old}}$. Applying \eqref{maintool} we thus obtain $|A_k| = |\E g(S_k)| \le \| f \|_1 \cdot \| f \|_{p-\textrm{H\"old}} \Delta_k$, and the absolute convergence of the series in \eqref{Cfnu} follows.

Next, let us expand the square in the claim:
\begin{equation}\label{expandsquarefusk}
\E \left( \sum_{k=1}^N f(US_k) \right)^2 = \sum_{k=1}^N \E f(US_k)^2 + 2 \sum_{1 \le k < \ell \le N} \E f(US_k) f(US_{\ell}) .
\end{equation}
Here $US_k$ is uniformly distributed, because $\mu * \nu^{*k}=\mu$. Let us write $US_{\ell}=US_k\prod_{j=k+1}^{\ell} X_j$, and observe that $\prod_{j=k+1}^{\ell} X_j$ has the same distribution as $S_{\ell -k}$, and is independent of $US_k$. Therefore $\E f(US_k)^2 = \E f(U)^2$ and $\E f(US_k) f(US_{\ell})=\E f(U)f(US_{\ell -k})$; in particular, \eqref{expandsquarefusk} simplifies as
\[ \begin{split} \E \left( \sum_{k=1}^N f(US_k) \right)^2 &= N \E f(U)^2 + 2 \sum_{1 \le k<\ell \le N} \E f(U)f(US_{\ell -k}) \\ &= N \E f(U)^2 + 2 \sum_{d=1}^{N-1} (N-d) \E f(U)f(US_d) \\ &= C(f,\nu ) N + O \left( \sum_{d=1}^{N-1} d |A_d| + N \sum_{d=N}^{\infty} |A_d| \right) . \end{split} \]
As observed before, here $|A_d| \le \| f \|_1 \cdot \| f \|_{p-\textrm{H\"old}} \Delta_d$. The error term in the previous estimate is thus $o(N)$.
\end{proof}

\begin{lem}\label{varianceupperbound} For any $f \in \mathcal{F}_p$ and any integer $N \ge 1$,
\[ \E \left( \sum_{k=1}^N f(S_k) \right)^2 \ll \| f \|_1 \cdot \| f \|_{p-\textrm{H\"old}} N + \| f \|_{p-\textrm{H\"old}}^2 \]
with an implied constant depending only on $\nu$, $p$ and the metric $d$.
\end{lem}

\begin{proof} We may assume that $\| f \|_{p-\textrm{H\"old}}=1$. Expanding the square we get
\begin{equation}\label{expandsquarefsk}
\E \left( \sum_{k=1}^N f(S_k) \right)^2 = \sum_{k=1}^N \E f(S_k)^2 + 2 \sum_{1 \le k <\ell \le N} \E f(S_k)f(S_{\ell}) .
\end{equation}
To estimate the first sum, note that $g(x)=f(x)^2$ satisfies
\[ \E g(U) = \| f \|_2^2 \le \| f \|_1 \cdot \| f \|_{\infty} \ll \| f \|_1 \cdot \| f \|_{p-\textrm{H\"old}} = \| f \|_1 \]
and
\[ \| g \|_{p-\textrm{H\"old}} \le 2 \| f \|_{\infty} \cdot \| f \|_{p-\textrm{H\"old}} \ll \| f \|_{p-\textrm{H\"old}}^2 = 1. \]
Hence by \eqref{maintool} we have
\begin{equation}\label{diagonalsum}
\sum_{k=1}^N \E f(S_k)^2 \ll \sum_{k=1}^N \left( \| f \|_1 + \Delta_k \right) \ll \| f \|_1 N +1 .
\end{equation}
Next, consider the second sum in \eqref{expandsquarefsk}. Since $S_{\ell}=S_k \prod_{j=k+1}^{\ell} X_j$, we have
\[ \E f(S_k) f(S_{\ell}) = \int_G \int_G f(x)f(xy) \, \mathrm{d} \nu^{*k}(x) \mathrm{d} \nu^{*(\ell -k)}(y) . \]
Let $g(y)=\int_G f(x)f(xy) \, \mathrm{d} \nu^{*k}(x)$ denote the inner integral, and observe $\E g(U)=\int_G g \, \mathrm{d} \mu =0$. From \eqref{maintool} we thus deduce $|\E f(S_k) f(S_{\ell})| \le \| g \|_{p-\textrm{H\"old}} \Delta_{\ell -k}$.

For any $y,y' \in G$ we have
\[ \begin{split} |g(y)-g(y')| &= \left| \int_G f(x) (f(xy)-f(xy')) \, \mathrm{d} \nu^{*k}(x) \right| \\ &\le \int_G |f(x)| \cdot \| f \|_{p-\textrm{H\"old}} d(xy,xy')^p \, \mathrm{d} \nu^{*k}(x) \\ &= \E |f(S_k)| \cdot d(y,y')^p \end{split} \]
where we used the translation invariance of the metric $d$. Applying \eqref{maintool} to $h(x)=|f(x)|$ we deduce $\E |f(S_k)| \le \E |f(U)| + \| f \|_{p-\textrm{H\"old}} \Delta_k = \| f \|_1 + \Delta_k$. Therefore $\| g \|_{p-\textrm{H\"old}} \le \| f \|_1 + \Delta_k$, and altogether we get
\[ |\E f(S_k)f(S_{\ell})| \le \left( \| f \|_1 + \Delta_k \right) \Delta_{\ell -k} . \]
Summing over $1 \le k<\ell \le N$ we finally obtain that the second sum in \eqref{expandsquarefsk} satisfies
\[ \sum_{1 \le k <\ell \le N} \E f(S_k)f(S_{\ell}) \ll \| f \|_1 N +1 . \]
This estimate together with \eqref{diagonalsum} finishes the proof.
\end{proof}

\subsection{Approximation by independent random variables}

The main goal of this section is to approximate $\sum_{k=1}^N f(S_k)$ by a sum of independent random variables. We use the same exact construction as in \cite{BO}; for the sake of completeness, we repeat the key details. This construction in its simplest form was first used by Schatte \cite{SCH1}, \cite{SCH2} on the circle group, and is admittedly somewhat ad hoc; we give an intuitive explanation first. We start by partitioning the set of positive integers into consecutive, nonempty, finite intervals of integers (called \textit{blocks} for short) $H_1, J_1, H_2, J_2, \dots$. Consider the block sums $Y_i=\sum_{k \in J_i} f(S_k)$ and $Z_i=\sum_{k \in H_i} f(S_k)$. In each block sum we perform a small perturbation of $S_k$: we replace the product corresponding to the preceding block by a uniformly distributed random variable. Thus e.g.\ in each term of $Y_i=\sum_{k \in J_i} f(S_k)$ we replace
\[ S_k = \left( \prod_{\ell \in H_1 \cup J_1 \cup \cdots \cup H_{i-1} \cup J_{i-1}} X_{\ell} \right) \prod_{\ell \in H_i} X_{\ell} \prod_{\substack{\ell \in J_i \\ \ell \le k}} X_{\ell} \]
by
\[ S_k^* = \left( \prod_{\ell \in H_1 \cup J_1 \cup \cdots \cup H_{i-1} \cup J_{i-1}} X_{\ell} \right) U_i \prod_{\substack{\ell \in J_i \\ \ell \le k}} X_{\ell}, \]
where $U_i$ is uniformly distributed. The reason why we use a uniformly distributed variable is that this way the perturbed block sums $Y_i^*=\sum_{k \in J_i} f(S_k^*)$ become \textit{independent}. Indeed, the uniform distribution has the peculiar ability to imply independence; this is related to translation invariance, a defining property of the Haar measure. For the technical details we refer to \cite[Lemmas 1 and 2]{BO}. To ensure that the error of this perturbation is small, we will use a coupling to construct $U_i$. If the size of $H_i$ is large, then the distribution of $\prod_{\ell \in H_i} X_{\ell}$ is close to the Haar measure, and thus replacing it by a carefully chosen uniformly distributed variable has a small effect; this makes the perturbation error $Y_i-Y_i^*$ small. Note that in the formal construction the sums $Y_i, Y_i^*, Z_i, Z_i^*$ will be defined in terms of variables $W_k$, $W_k^*$ which only have the same distribution as $S_k$, $S_k^*$ but are not actually equal to them; however, this is just a technicality.

The perturbation method thus approximates $\sum_{k=1}^N f(S_k)$ by two sums of independent variables $Y_1^*, Y_2^*, \dots$ and $Z_2^*, Z_3^*, \dots$; the two sums, however, are not independent of one another. To overcome this problem, we need to make sure that, say, the variables $Z_i^*$ are negligible compared to the variables $Y_i^*$; this way $\sum_{k=1}^N f(S_k)$ is approximated by a single sum of independent variables, and limit theorems such as the CLT and the LIL follow. Therefore our main strategy is to choose the sizes of the blocks $H_1, J_1, H_2, J_2, \dots$ optimally: on the one hand, we need $H_i$ to be large enough to make the perturbation error small; on the other hand, we need $H_i$ to be small enough so that $Z_i^*$ is negligible compared to $Y_i^*$.

We now give the formal construction. By the definition of the $p$-Wasserstein metric, for any $k \ge 1$ there exists a coupling $\vartheta_k$ of $\nu^{*k}$ and the Haar measure $\mu$ such that, say\footnote{Since $G$ is compact, the infimum in the definition of $W_p$ is actually attained by some coupling; however, we do not need this fact.},
\[ \int_{G \times G} d(x,y)^p \, \mathrm{d} \vartheta_k (x,y) \le 2 W_p (\nu^{*k}, \mu ) = 2 \Delta_k . \]
For any nonempty, finite interval of positive integers $J \subseteq \mathbb{N}$ let $S_J = \prod_{j \in J} X_j$. After a suitable extension of the probability space, for any such $J \subseteq \mathbb{N}$ we can introduce a pair of $G$-valued random variables $(T_J, U_J)$ with joint distribution $\vartheta_{|J|}$, where $|J|$ denotes the cardinality of $J$. That is, $T_J \overset{d}{=} S_J$, the variable $U_J$ is uniformly distributed, and $\E \, d(T_J, U_J)^p \le 2 \Delta_{|J|}$. Further, we may assume that $X_1, X_2, \dots$ and $(T_J, U_J)$, $J \subseteq \mathbb{N}$ are independent.

Fix a block decomposition $H_1, J_1, H_2, J_2, \dots$ of the set of positive integers. For any $i \ge 1$ and $k \in J_i$ let
\[ W_k = \prod_{j=1}^{i-1} \left( T_{H_j} S_{J_j} \right) T_{H_i} \prod_{\substack{\ell \in J_i \\ \ell \le k}} X_{\ell}, \quad W_k^* = \prod_{j=1}^{i-1} \left( T_{H_j} S_{J_j} \right) U_{H_i} \prod_{\substack{\ell \in J_i \\ \ell \le k}} X_{\ell} . \]
Similarly, for any $i \ge 2$ and $k \in H_i$ let
\[ W_k = \prod_{j=1}^{i-2} \left( S_{H_j} T_{J_j} \right) S_{H_{i-1}} T_{J_{i-1}} \prod_{\substack{\ell \in H_i \\ \ell \le k}} X_{\ell}, \quad W_k^* = \prod_{j=1}^{i-2} \left( S_{H_j} T_{J_j} \right) S_{H_{i-1}} U_{J_{i-1}} \prod_{\substack{\ell \in H_i \\ \ell \le k}} X_{\ell} . \]
We use the convention that an empty product equals the identity element $1 \in G$. Note that $W_k^*$ is obtained from $W_k$ by replacing a block product by a uniformly distributed variable; consequently, $W_k^*$ is uniformly distributed for all $k$. Observe that the (joint) distribution of the sequence $W_k$, $k \in \bigcup_{i=1}^{\infty} J_i$ is the same as the (joint) distribution of the sequence $S_k$, $k \in \bigcup_{i=1}^{\infty} J_i$; the same holds for the sequences $W_k$, $k \in \bigcup_{i=2}^{\infty} H_i$ and $S_k$, $k \in \bigcup_{i=2}^{\infty} H_i$. Finally, for a given $f \in \mathcal{F}_p$ let us introduce the random variables
\[ \begin{split} Y_i &= \sum_{k \in J_i} f(W_k), \qquad Y_i^* = \sum_{k \in J_i} f(W_k^*) \quad (i \ge 1), \\ Z_i &= \sum_{k \in H_i} f(W_k), \qquad Z_i^* = \sum_{k \in H_i} f(W_k^*) \quad (i \ge 2). \end{split} \]

Observe that along the sequence $N_R = \max J_R$ ($R=1,2,\dots$) we have
\[ \sum_{k=1}^{N_R} f(S_k) = \sum_{k \in H_1} f(S_k) + \sum_{i=1}^R \sum_{k \in J_i} f(S_k) + \sum_{i=2}^R \sum_{k \in H_i} f(S_k) , \]
where the (joint) distribution of the sequence $\sum_{i=1}^R \sum_{k \in J_i} f(S_k)$ ($R=1,2,\dots$) is the same as the (joint) distribution of the sequence $\sum_{i=1}^R Y_i$ ($R=1,2,\dots$), and the same holds for the sequences $\sum_{i=2}^R \sum_{k \in H_i} f(S_k)$ and $\sum_{i=2}^R Z_i$.

The key observation is that replacing $Y_i$ by $Y_i^*$, and $Z_i$ by $Z_i^*$ introduces \textit{independence}; for a detailed proof see \cite[Lemmas 1 and 2]{BO}. Also note that $Y_i^*\overset{d}{=}\sum_{k=1}^{|J_i|} f(US_k)$ and $Z_i^*\overset{d}{=}\sum_{k=1}^{|H_i|} f(US_k)$, thus $Y_i^*$ and $Z_i^*$ are mean zero random variables with variance given by Lemma \ref{variancefusk}. The main properties of the approximating variables are summarized as follows.
\begin{lem}\label{yi*zi*lemma} Let $f \in \mathcal{F}_p$, and let $H_1, J_1, H_2, J_2, \dots$ be an arbitrary block decomposition of $\mathbb{N}$. Then $Y_1^*, Y_2^*, \dots$ are independent, $\E Y_i^* =0$, and $\E (Y_i^*)^2 = C(f, \nu ) |J_i| + o(|J_i|)$. The same holds for $Z_2^*, Z_3^*, \dots$ with $|J_i|$ replaced by $|H_i|$.
\end{lem}
Next, we estimate the perturbation error. The main idea is that in $Y_i-Y_i^*=\sum_{k \in J_i} (f(W_k)-f(W_k^*))$ the variable $W_k^*$ is obtained from $W_k$ by replacing $T_{H_i}$ by $U_{H_i}$. Since $\E \, d(T_{H_i}, U_{H_i})^p \le 2 \Delta_{|H_i|}$ by construction, the error cannot be too large.
\begin{lem}\label{approxlemma} Let $f \in \mathcal{F}_p$, and let $H_1, J_1, H_2, J_2, \dots$ be an arbitrary block decomposition of $\mathbb{N}$. For any integers $1 \le R<S$,
\[ \begin{split} \E \left( \sum_{i=R+1}^S (Y_i-Y_i^*) \right)^2 &\ll \| f \|_{p-\textrm{H\"old}}^2 \sum_{i=R+1}^S \left( \Delta_{|H_i|} |J_i| +1 \right), \\ \E \left( \sum_{i=R+1}^S (Z_i-Z_i^*) \right)^2 &\ll \| f \|_{p-\textrm{H\"old}}^2 \sum_{i=R+1}^S \left( \Delta_{|J_i|} |H_i| +1 \right)  \end{split} \]
with implied constants depending only on $\nu$, $p$ and the metric $d$.
\end{lem}

\begin{proof} We only give a proof of the first estimate; the proof of the second one is exactly the same. We may assume that $\| f \|_{p-\textrm{H\"old}} =1$. We need to estimate
\begin{equation}\label{expandsquareerror}
\E \left( \sum_{i=R+1}^S (Y_i-Y_i^*) \right)^2 = \sum_{i=R+1}^S \E (Y_i-Y_i^*)^2 + 2 \sum_{R+1\le i<j\le S} \E (Y_i-Y_i^*)(Y_j-Y_j^*).
\end{equation}
First, fix $i \ge 2$ and consider $Y_i-Y_i^*=\sum_{k \in J_i} (f(W_k)-f(W_k^*))$. By construction, here $W_k$ and $W_k^*$ are of the form $W_k=A T_{H_i}\prod_{\ell \in J_i, \ell \le k} X_{\ell}$ and $W_k^*=A U_{H_i}\prod_{\ell \in J_i, \ell \le k} X_{\ell}$, where $A$ is independent of the other factors. Let $\alpha$ denote the distribution of $A$ (in fact, it is a convolution power of $\nu$), and recall that $(T_{H_i},U_{H_i})$ has joint distribution $\vartheta_{|H_i|}$. By the independence of the factors we have
\[ \E (Y_i-Y_i^*)^2 = \int_G \int_{G \times G} \E \left( \sum_{k=1}^{|J_i|} \left( f(xyS_k)-f(xy'S_k) \right) \right)^2 \, \mathrm{d} \vartheta_{|H_i|}(y,y') \mathrm{d} \alpha (x) . \]
Observe that for any given $x,y,y' \in G$, the function $g(z)=f(xyz)-f(xy'z)$ satisfies $\int_G g \, \mathrm{d} \mu =0$, and $\| g \|_{p-\textrm{H\"old}} \le 2 \| f \|_{p-\textrm{H\"old}}=2$. Further,
\[ \begin{split} \| g \|_1 &= \int_G \left| f(xyz)-f(xy'z) \right| \, \mathrm{d} \mu (z) \\ &\le \| f \|_{p-\textrm{H\"old}} \int_G d(xyz,xy'z)^p \, \mathrm{d} \mu (z) \\ &= d(y,y')^p \end{split} \]
by the bi-invariance of the metric $d$. Applying Lemma \ref{varianceupperbound} to $g$, we thus get
\[ \E \left( \sum_{k=1}^{|J_i|} \left( f(xyS_k)-f(xy'S_k) \right) \right)^2 \ll d(y,y')^p |J_i| +1 \]
uniformly in $x, y, y' \in G$. By the construction of $\vartheta_{|H_i|}$ we finally deduce
\begin{equation}\label{diagonalyiyi*}
\begin{split} \E (Y_i-Y_i^*)^2 &\ll \int_G \int_{G \times G} \left( d(y,y')^p |J_i| +1 \right) \, \mathrm{d} \vartheta_{|H_i|}(y,y') \mathrm{d} \alpha (x) \\ &\ll \Delta_{|H_i|} |J_i| +1 . \end{split}
\end{equation}

Next, fix $2 \le i<j$ and consider $Y_j-Y_j^* = \sum_{k \in J_j} (f(W_k)-f(W_k^*))$. By construction, here $W_k$ and $W_k^*$ can be written in the form $W_k=BC_k$ and $W_k^*=BC_k^*$, where $B=T_{H_1} S_{J_1} \cdots T_{H_i} S_{J_i}$. Let $\mathcal{G}_i$ be the $\sigma$-algebra generated by the variables $X_{\ell}$, $\ell \in J_1 \cup \dots \cup J_i$, the variables $T_{H_1}, \dots, T_{H_i}$, and $U_{H_i}$. Then $Y_i$, $Y_i^*$ and $B$ are $\mathcal{G}_i$-measurable, and $C_k$ and $C_k^*$ are independent of $\mathcal{G}_i$. Therefore
\[ \left| \E \left( (Y_i-Y_i^*)(Y_j-Y_j^*) \mid \mathcal{G}_i \right) \right| = |Y_i-Y_i^*| \left| \sum_{k \in J_j} \E \left( f(BC_k)-f(BC_k^*) \mid \mathcal{G}_i \right) \right| . \]
Since $C_k^*$ is uniformly distributed and $\E f(bC_k^*)=0$ for all $b \in G$, here $\E \left( f(BC_k^*) \mid \mathcal{G}_i \right) =0$. A trivial estimate gives $\left| \E \left( f(BC_k) \mid \mathcal{G}_i \right) \right| \le \sup_{b \in G} |\E f(bC_k)|$. Note that $C_k$ has distribution $\nu^{*m}$ with $m=k-\max J_i$. Applying \eqref{maintool} to $g(x)=f(bx)$ we thus obtain $|\E f(bC_k)| \le \Delta_{k-\max J_i}$ for any $b \in G$. Hence
\[ \left| \E \left( (Y_i-Y_i^*)(Y_j-Y_j^*) \mid \mathcal{G}_i \right) \right| \le |Y_i-Y_i^*| \sum_{k\in J_j} \Delta_{k-\max J_i} . \]
By taking the (total) expectation and summing over $j=i+1, i+2, \dots$,
\[ \sum_{j=i+1}^{\infty} \left| \E (Y_i-Y_i^*)(Y_j-Y_j^*) \right| \le \E |Y_i-Y_i^*| \sum_{j=i+1}^{\infty} \sum_{k \in J_j} \Delta_{k-\max J_i} \ll \E |Y_i-Y_i^*|. \]
This estimate together with \eqref{diagonalyiyi*} shows that in \eqref{expandsquareerror} we have
\[ \sum_{i=R+1}^S \E (Y_i-Y_i^*)^2 + 2 \sum_{R+1\le i<j\le S} \E (Y_i-Y_i^*)(Y_j-Y_j^*) \ll \sum_{i=R+1}^S \left( \Delta_{|H_i|} |J_i| +1 \right) \]
for any $1 \le R<S$, as claimed.
\end{proof}

\subsection{Approximation by a Wiener process}

In this section we finish the proof of Theorem \ref{theoremWIENER}, and prove the Remark thereafter.

\begin{proof}[Proof of Theorem \ref{theoremWIENER}]
Fix $f \in \mathcal{F}_p$, and consider the block decomposition $H_1, J_1, H_2, J_2, \dots$ of $\mathbb{N}$ with $|H_i|=\lfloor i^{1/4} \rfloor$ and $|J_i|=\lfloor i^{1/2} \rfloor$. For any $t \ge 1$ let $R(t)$ denote the positive integer for which $\lfloor t \rfloor \in H_{R(t)} \cup J_{R(t)}$; in particular, $R(t) = \Theta (t^{2/3})$. Replacing $t$ by $\max J_{R(t)}$ in the sum $\sum_{k \le t} f(S_k)$, we thus have
\begin{equation}\label{sumfskdecomposition}
\sum_{k \le t} f(S_k) = \sum_{i=1}^{R(t)} \sum_{k \in J_i} f(S_k) + \sum_{i=2}^{R(t)} \sum_{k \in H_i} f(S_k) + O \left( t^{1/3} \right) .
\end{equation}
The first double sum has the same distribution as $\sum_{i=1}^{R(t)} Y_i$, whereas the second double sum has the same distribution as $\sum_{i=2}^{R(t)} Z_i$. Recall that $k\Delta_k \to 0$ as $k \to \infty$; in particular, $\Delta_{|H_i|}|J_i|+1 \ll i^{1/4}$ and $\Delta_{|J_i|}|H_i|+1 \ll 1$.

Lemma \ref{approxlemma} and the Rademacher--Menshov inequality \cite[Theorem F]{MO} give that for any integer $m \ge 1$,
\[ \E \left( \max_{1 \le R \le 2^m} \left| \sum_{i=1}^R (Y_i-Y_i^*) \right| \right)^2 \ll m^2 \sum_{i=1}^{2^m} \left( \Delta_{|H_i|}|J_i|+1 \right) \ll m^2 2^{5m/4} . \]
By the Chebyshev inequality and the Borel--Cantelli lemma, for any $\varepsilon >0$
\[ \max_{1 \le R \le 2^m} \left| \sum_{i=1}^R (Y_i-Y_i^*) \right| \ll 2^{(5/8+\varepsilon)m} \,\, \mathrm{a.s.} \]
Repeating the same arguments for $Z_i-Z_i^*$, we deduce
\begin{equation}\label{yiyi*zizi*almostsure}
\left| \sum_{i=1}^R (Y_i-Y_i^*) \right| \ll R^{5/8+\varepsilon} \,\, \mathrm{a.s.} \quad \mathrm{and} \quad \left| \sum_{i=2}^R (Z_i-Z_i^*) \right| \ll R^{1/2+\varepsilon} \,\, \mathrm{a.s.}
\end{equation}
Recall from Lemma \ref{yi*zi*lemma} that $Z_2^*, Z_3^*, \dots$ are independent, mean zero random variables with $\E (Z_i^*)^2 \ll |H_i| \le i^{1/4}$. In particular, $\sum_{i=1}^{\infty} \E (Z_i^*)^2 / i^{(5/8+\varepsilon)2}< \infty$, and by the strong law of large numbers $\left| \sum_{i=1}^R Z_i^* \right| \ll R^{5/8+\varepsilon}$ a.s.\ for all $\varepsilon >0$. This estimate, together with \eqref{yiyi*zizi*almostsure} show that the processes $\sum_{k \le t} f(S_k)$ and $\sum_{i=1}^{R(t)} Y_i^*$ are $o(R(t)^{5/8+\varepsilon})$-equivalent, i.e.\ $o(t^{5/12+\varepsilon})$-equivalent for all $\varepsilon>0$. It will thus be enough to prove the claims for $\sum_{i=1}^{R(t)} Y_i^*$ instead of $\sum_{k \le t} f(S_k)$.

Recall from Lemma \ref{yi*zi*lemma} that $Y_1^*, Y_2^*, \dots$ are independent, mean zero random variables. We start with the easier case $C(f,\nu)=0$. In this case $\sum_{i=1}^R \E (Y_i^*)^2 = \sum_{i=1}^R o(|J_i|)=o(R^{3/2})$, and $|Y_i^*| \ll |J_i| \le i^{1/2}$ by construction. From a general form of the strong law of large numbers and Kolmogorov's exponential bounds it follows that $\sum_{i=1}^R Y_i^* = o(\sqrt{R^{3/2} \log \log R})$ a.s.; for a detailed proof see the almost sure stability criterion in Lo\`eve \cite[p.\ 270]{LO}. Therefore $\sum_{i=1}^{R(t)} Y_i^* = o(\sqrt{t \log \log t})$ a.s., and consequently $\sum_{k \le t} f(S_k)=o(\sqrt{t \log \log t})$ a.s. Further, we have $\sum_{i=1}^{R(t)} Y_i^*=o(R(t)^{3/4})=o(t^{1/2})$ in $L^2$, and consequently $\sum_{k \le t} f(S_k)=o(t^{1/2})$ in $L^2$. In particular, $t^{-1/2}\sum_{k \le t} f(S_k) \overset{d}{\to} 0$.

Next, assume $C(f,\nu)>0$. It remains to find an almost sure approximation of $\sum_{i=1}^{R(t)} Y_i^*$ by a Wiener process. In fact, we will prove that $\sigma(t)^2 = \sum_{i=1}^{R(t)} \E (Y_i^*)^2$ satisfies the claims of Theorem \ref{theoremWIENER}. First, note that by Lemma \ref{yi*zi*lemma} we have
\[ \sigma(t)^2 = \sum_{i=1}^{R(t)} \left( C(f, \nu ) |J_i| + o(|J_i|) \right) = C(f,\nu ) t +o(t) . \]
Now let $V_R=\sum_{i=1}^R \E (Y_i^*)^2=\Theta (R^{3/2})$, $R=1,2,\dots$, and for any $t \ge V_1$ let $R'(t)$ be the positive integer for which $V_{R'(t)} \le t < V_{R'(t)+1}$. An invariance principle proved by Strassen \cite[Theorem 4.4]{STR} implies that if $\sum_{i=1}^{\infty} \E |Y_i^*|^q / V_i^{\theta q /2} < \infty$ for some $q>2$ and $0 \le \theta \le 1$, then the stochastic processes $\sum_{i=1}^{R'(t)} Y_i^*$ and $W(t)$ are $o(t^{(1+\theta)/4} \log t)$-equivalent, where $W(t)$ is a standard Wiener process. By construction $|Y_i^*| \ll |J_i| \le i^{1/2}$, and so $\E |Y_i^*|^q \ll i^{q/2}$. Therefore for any $\theta>2/3$ there exists a large enough $q>2$ such that
\[ \sum_{i=1}^{\infty} \frac{\E |Y_i^*|^q}{V_i^{\theta q/2}} \ll \sum_{i=1}^{\infty} i^{q(1/2 - 3\theta /4)} < \infty . \]
Hence by the Strassen invariance principle $\sum_{i=1}^{R'(t)} Y_i^*$ and $W(t)$ are $o(t^{5/12+\varepsilon})$-equivalent for all $\varepsilon >0$. Using $\sigma(t)^2 \sim C(f,\nu ) t$, it is not difficult to see that $\sum_{i=1}^{R'(\sigma (t)^2)} Y_i^*$ and $W(\sigma (t)^2)$ are also $o(t^{5/12+\varepsilon})$-equivalent for all $\varepsilon >0$. Since $R'(\sigma (t)^2)=R(t)$ for all $t$, the last relation is exactly what we wanted to prove.
\end{proof}

We now prove the remark made after Theorem \ref{theoremWIENER}. If $\Delta_k \ll k^{-(1+c)}$ for some $0<c \le 1/2$, then in Lemma \ref{variancefusk} the error term $o(N)$ can be replaced by $O(N^{1-c})$. In particular, in Lemma \ref{yi*zi*lemma} we have $\E (Y_i^*)^2=C(f,\nu) |J_i| + O(|J_i|^{1-c})$, and so
\[ \begin{split} \sigma (t)^2 &= \sum_{i=1}^{R(t)} \left( C(f,\nu ) |J_i| + O(|J_i|^{1-c}) \right) \\ &= C(f,\nu ) \sum_{i=1}^{R(t)} (|J_i|+|H_i|) + O\left( R(t)^{5/4} + R(t)^{3/2-c/2} \right) \\ &= C(f,\nu )t + O\left( t^{1-c/3} \right) . \end{split} \]
Finally, we show that in this case $W(\sigma (t)^2)$ and $\sqrt{C(f,\nu)} W(t)$ are $o(t^{1/2-c/6+\varepsilon})$-equivalent for all $\varepsilon >0$. Recalling the distribution of the running maximum of a Wiener process, for an arbitrarily large constant $K>0$ and any $\varepsilon >0$,
\[ \Pr \left( \sup_{s \in [n-Kn^{1-c/3}, n+Kn^{1-c/3}]} |W(s)-W(n)| \ge n^{1/2-c/6+\varepsilon /2} \right) \ll \frac{1}{n^2} . \]
The Borel--Cantelli lemma thus shows that $W(\sigma (t)^2)$ and $W(C(f,\nu) t)$ are both $o(t^{1/2-c/6+\varepsilon})$-equivalent to, say, $W(\lfloor C(f,\nu) t \rfloor)$. By the scaling property of the Wiener process, $W(C(f,\nu) t)$ and $\sqrt{C(f,\nu)} W(t)$ have the same distribution.

\section{Random walks on the torus}\label{torussection}

For any positive integer $H$, let
\[ F_H(x)=\sum_{|h|<H} \left( 1-\frac{|h|}{H} \right) e^{2 \pi i h x} = \frac{1}{H} \cdot \frac{\sin^2 (\pi H x)}{\sin^2 (\pi x)} \]
denote the $1$-dimensional Fej\'er kernel of order $H$, and let $F_H^{(d)}(x_1, x_2, \dots, x_d) = F_H(x_1) F_H(x_2) \cdots F_H(x_d)$ denote its $d$-dimensional counterpart. Recall that given a function $f \in L^1 (\mathbb{R}^d/\mathbb{Z}^d)$, the convolution $f_H=f*F_H^{(d)}$ is a Ces\`aro average of the formal Fourier series of $f$; more precisely, $f_H$ is a trigonometric polynomial whose Fourier coefficients are $\widehat{f_H}(h)=(\prod_{j=1}^d (1-|h_j|/H)) \widehat{f}(h)$ if $\| h \|_{\infty}<H$, and $\widehat{f_H}(h)=0$ otherwise.

Since
\[ \int_0^1 F_H(x)^2 \, \mathrm{d}x = \sum_{|h|<H} \left( 1-\frac{|h|}{H} \right)^2 = \frac{2H^2+1}{3H}, \]
the normalized square of the Fej\'er kernel --- the so-called Jackson kernel --- is
\[ K_H (x) = \frac{3}{2H^3+H} \cdot \frac{\sin^4 (\pi H x)}{\sin^4 (\pi x)} . \]
Note that its Fourier series is of the form $K_H (x)=\sum_{|h|<2H-1} a_h e^{2 \pi i h x}$, where the coefficients $a_h$ are symmetric and unimodal; in particular, $0 \le a_h \le a_0=1$ for all $h$. We also introduce the $d$-dimensional version $K_H^{(d)}(x_1, x_2, \dots, x_d)=K_H (x_1) K_H(x_2) \cdots K_H (x_d)$. Given $f \in L^1 (\mathbb{R}^d / \mathbb{Z}^d)$, the convolution $f_H=f*K_H^{(d)}$ is again a trigonometric polynomial whose Fourier coefficients are $\widehat{f_H}(h)=(\prod_{j=1}^d a_{h_j}) \widehat{f}(h)$ if $\| h \|_{\infty}<2H-1$, and $\widehat{f_H}(h)=0$ otherwise.

\begin{proof}[Proof of Proposition \ref{sigmaequalzero}] We will prove that under the hypotheses of Theorem \ref{theoremCLTLIL} for any $f \in \mathcal{F}_p$,
\begin{equation}\label{cfnualternative}
\sigma^2 = C(f,\nu ) = \sum_{\substack{h \in \mathbb{Z}^d \\ h \neq 0}} |\widehat{f}(h)|^2 \frac{1-|\widehat{\nu}(h)|^2}{|1-\widehat{\nu}(h)|^2} ,
\end{equation}
where $C(f, \nu)$ is as in \eqref{Cfnu}. Here $|\widehat{\nu}(h)|<1$ for every $h \neq 0$. Indeed, by the assumption $W_p(\nu^{*k}, \mu) \to 0$ we have $\nu^{*k} \to \mu$ weakly as $k \to \infty$, and hence
\[ \widehat{\nu}(h)^k = \widehat{\nu^{*k}}(h) = \int_{\mathbb{R}^d / \mathbb{Z}^d} e^{-2 \pi i \langle h,x \rangle} \, \mathrm{d} \nu^{*k}(x) \to \int_{\mathbb{R}^d / \mathbb{Z}^d} e^{-2 \pi i \langle h,x \rangle} \, \mathrm{d} \mu (x) = 0. \]
In particular, \eqref{cfnualternative} implies that $\sigma =0$ if and only if $\widehat{f}(h)=0$ for all $h \in \mathbb{Z}^d$. The latter condition is equivalent to $f=0$ a.e., and by the continuity of $f$, to $f=0$.

We establish \eqref{cfnualternative} in two steps. First, assume that $f$ is a trigonometric polynomial; that is, assume $\widehat{f}(h) \neq 0$ for finitely many $h \in \mathbb{Z}^d$. Since the variable $U$ in the definition \eqref{Cfnu} is independent of $X_1, X_2, \dots$,
\[ \E f(U)f(U+S_k) = \int_{\mathbb{R}^d / \mathbb{Z}^d} \int_{\mathbb{R}^d / \mathbb{Z}^d} f(u)f(u+x) \, \mathrm{d} \mu (u) \mathrm{d} \nu^{*k}(x) = \E g(S_k) , \]
where $g(x)=\int_{\mathbb{R}^d / \mathbb{Z}^d} f(u)f(u+x) \, \mathrm{d} \mu (u)$. Since $\widehat{g}(h)=|\widehat{f}(h)|^2$ and $g$ is also a trigonometric polynomial,
\[ \E g(S_k) = \sum_{h \in \mathbb{Z}^d} \widehat{g}(h) \E e^{2 \pi i \langle h, S_k \rangle} = \sum_{h \in \mathbb{Z}^d} |\widehat{f}(h)|^2 \widehat{\nu}(h)^k . \]
Using $\widehat{f}(0)=0$ and the Parseval formula $\E f(U)^2 = \sum_{h \in \mathbb{Z}^d} |\widehat{f}(h)|^2$, in \eqref{Cfnu} we thus have
\[ C(f,\nu ) = \E f(U)^2 + 2 \sum_{k=1}^{\infty} \E g(S_k) = \sum_{\substack{h \in \mathbb{Z}^d \\ h \neq 0}} |\widehat{f}(h)|^2 \left( 1 + 2 \frac{\widehat{\nu}(h)}{1-\widehat{\nu}(h)} \right) . \]
Combining the $h$ and $-h$ terms, \eqref{cfnualternative} follows.

Finally, let us show \eqref{cfnualternative} for arbitrary $f \in \mathcal{F}_p$. Consider the convolutions $f_H=f*F_H^{(d)}$, $H=1,2,\dots$. By the special case shown above,
\[ C(f_H, \nu ) = \sum_{\substack{h \in \mathbb{Z}^d \\ 0< \| h \|_{\infty} <H}} \prod_{j=1}^d \left( 1 -\frac{|h_j|}{H} \right)^2 |\widehat{f}(h)|^2 \frac{1-|\widehat{\nu}(h)|^2}{|1-\widehat{\nu}(h)|^2} \to \sum_{\substack{h \in \mathbb{Z}^d \\ h \neq 0}} |\widehat{f}(h)|^2 \frac{1-|\widehat{\nu}(h)|^2}{|1-\widehat{\nu}(h)|^2} \]
as $H \to \infty$ (seen e.g.\ from the monotone convergence theorem). Now let $g(x)=\int_{\mathbb{R}^d / \mathbb{Z}^d} f(u)f(u+x) \, \mathrm{d} \mu (u)$ and $g_H(x)=\int_{\mathbb{R}^d / \mathbb{Z}^d} f_H(u)f_H(u+x) \, \mathrm{d} \mu (u)$. Applying \eqref{maintool},
\[ \begin{split} \left| C(f,\nu ) - C(f_H, \nu ) \right| &= \left| \E f(U)^2 + 2 \sum_{k=1}^{\infty} \E g(S_k) - \E f_H(U)^2 - 2 \sum_{k=1}^{\infty} \E g_H(S_k) \right| \\ &\le |\E f(U)^2 - \E f_H(U)^2| + 2 \sum_{k=1}^{\infty} \left| \E (g-g_H)(S_k) \right| \\ &\le \left| \| f \|_2^2 - \| f_H \|_2^2 \right| + 2 \sum_{k=1}^{\infty} \| g-g_H \|_{p-\textrm{H\"old}} W_p (\nu^{*k}, \mu ) \\ & \ll \left| \| f \|_2^2 - \| f_H \|_2^2 \right| + \| g-g_H \|_{p-\textrm{H\"old}} .  \end{split} \]
Since $f_H \to f$ in $L^2(\mathbb{R}^d/\mathbb{Z}^d)$, the first term goes to $0$ as $H \to \infty$. It will thus be enough to show $\| g-g_H \|_{p-\textrm{H\"old}} \to 0$ as $H \to \infty$; indeed, this will imply $C(f_H,\nu) \to C(f,\nu)$, and \eqref{cfnualternative} will follow.

Writing $f=f_H+(f-f_H)$ in the definition of $g$, and applying the integral transformation $u \mapsto u-x$ in one of the terms, we get
\[ g(x)-g_H(x)= \int_{\mathbb{R}^d / \mathbb{Z}^d} (f-f_H)(u) \left( f(u+x) +f_H (u-x) \right) \, \mathrm{d} \mu (u) \]
for all $x \in \mathbb{R}^d / \mathbb{Z}^d$. Therefore $\| g-g_H \|_{p-\textrm{H\"old}} \le \| f-f_H \|_1 \left( \| f \|_{p-\textrm{H\"old}} + \| f_H \|_{p-\textrm{H\"old}} \right)$. Using the fact that $F_H^{(d)}$ is a nonnegative kernel, it is not difficult to see that $\| f_H \|_{p-\textrm{H\"old}} \le \| f \|_{p-\textrm{H\"old}}$. Hence $\| g-g_H \|_{p-\textrm{H\"old}} \le \| f-f_H \|_1 \cdot 2 \| f \|_{p-\textrm{H\"old}} \to 0$ as $H \to \infty$. This finishes the proof of \eqref{cfnualternative} for arbitrary $f \in \mathcal{F}_p$.
\end{proof}

\begin{proof}[Proof of Proposition \ref{ErdosTuran}] Let $f \in \mathcal{F}_1$ be such that $\| f \|_{1-\textrm{H\"old}} \le 1$, and let $H \ge 1$ be an integer. Consider the convolution $f_H=f*K_H^{(d)}$. Clearly,
\begin{equation}\label{kantorovichupperbound}
\begin{split} \bigg| \int_{\mathbb{R}^d/\mathbb{Z}^d} f \, \mathrm{d} \nu_1 &- \int_{\mathbb{R}^d/\mathbb{Z}^d} f \, \mathrm{d} \nu_2 \bigg| \\ &\le 2 \| f-f_H \|_{\infty} + \left| \int_{\mathbb{R}^d/\mathbb{Z}^d} f_H \, \mathrm{d} \nu_1 - \int_{\mathbb{R}^d/\mathbb{Z}^d} f_H \, \mathrm{d} \nu_2 \right| \\ &= 2 \| f-f_H \|_{\infty} + \left| \sum_{h \in \mathbb{Z}^d} \widehat{f_H}(h) \left( \widehat{\nu_1}(-h)-\widehat{\nu_2}(-h) \right) \right| \\ &\le 2 \| f-f_H \|_{\infty} + \sum_{\substack{h \in \mathbb{Z}^d \\ 0<\| h \|_{\infty}<2H-1}} |\widehat{f}(h)| \cdot |\widehat{\nu_1}(h)-\widehat{\nu_2}(h)| . \end{split}
\end{equation}
The Jackson approximation theorem gives an upper bound for $\| f-f_H \|_{\infty}$; for the sake of completeness, we include the short proof. We have
\[ \begin{split} \| f-f_H\|_{\infty} &= \sup_{x \in \mathbb{R}^d / \mathbb{Z}^d} \left| \int_{\mathbb{R}^d/\mathbb{Z}^d} (f(x)-f(x-y))K_H^{(d)}(y) \, \mathrm{d}\mu (y) \right| \\ &\le \int_{[-1/2,1/2)^d} |y| K_H^{(d)}(y) \, \mathrm{d}y \\ &\le \int_{-1/2}^{1/2} \cdots \int_{-1/2}^{1/2} \left( |y_1|+\cdots + |y_d| \right) K_H(y_1) \cdots K_H(y_d) \, \mathrm{d}y_1 \cdots \mathrm{d}y_d \\ &= d \int_{-1/2}^{1/2} |y| K_H(y) \, \mathrm{d}y . \end{split} \]
Note that the denominator of $K_H(y)$ satisfies $\sin^4 (\pi y) \ge 16 y^4$ on $[-1/2,1/2]$. Applying the integral transformation $u=Hy$ and then extending the range of integration,
\begin{equation}\label{f-fHsupbound}
\begin{split} \| f-f_H \|_{\infty} &\le \frac{3d}{2H^3+H} \int_{-1/2}^{1/2} \frac{\sin^4 (\pi H y)}{16|y|^3} \, \mathrm{d}y \\ &\le \frac{3 d}{32H+16/H} \int_{-\infty}^{\infty} \frac{\sin^4 (\pi u)}{|u|^3} \, \mathrm{d}u \\ &= \frac{3 d}{32H+16/H} 2 \pi^2 \log 2. \end{split}
\end{equation}
Next, we estimate the second term in \eqref{kantorovichupperbound}. By the Parseval formula and the assumption that $f$ is Lipschitz, for any $y \in \mathbb{R}^d / \mathbb{Z}^d$ we have
\[ \sum_{h \in \mathbb{Z}^d} |\widehat{f}(h)|^2 \cdot |e^{2 \pi i \langle h,y \rangle} -1|^2 = \int_{\mathbb{R}^d / \mathbb{Z}^d} |f(x+y)-f(x)|^2 \, \mathrm{d} \mu (x) \le \| y \|_{\mathbb{R}^d / \mathbb{Z}^d}^2 . \]
By comparing the asymptotics of the left and the right hand side as $y$ approaches the origin along the $j$th coordinate axis, we get $\sum_{h \in \mathbb{Z}^d} |\widehat{f}(h)|^2 4 \pi^2 h_j^2 \le 1$; summing over $j$, we obtain the general Fourier decay estimate for Lipschitz functions
\[ \sum_{h \in \mathbb{Z}^d} |\widehat{f}(h)|^2 \cdot |h|^2 \le \frac{d}{4 \pi^2} . \]
Applying the Cauchy--Schwarz inequality and this Fourier decay estimate, the second term in \eqref{kantorovichupperbound} satisfies
\[ \sum_{\substack{h \in \mathbb{Z}^d \\ 0<\| h \|_{\infty}<2H-1}} |\widehat{f}(h)| \cdot |\widehat{\nu_1}(h)-\widehat{\nu_2}(h)| \le \frac{d^{1/2}}{2 \pi} \left( \sum_{\substack{h \in \mathbb{Z}^d \\ 0< \| h \|_{\infty}<2H-1}} \frac{|\widehat{\nu_1}(h) - \widehat{\nu_2}(h)|^2}{|h|^2} \right)^{1/2} . \]
The previous line and \eqref{f-fHsupbound} provide an upper bound in \eqref{kantorovichupperbound} which does not depend on $f$; in particular, by Kantorovich duality,
\[ W_1 (\nu_1, \nu_2 ) \le \frac{(3 \pi^2 \log 2) d}{8H+4/H} + \frac{d^{1/2}}{2 \pi} \left( \sum_{\substack{h \in \mathbb{Z}^d \\ 0< \| h \|_{\infty}<2H-1}} \frac{|\widehat{\nu_1}(h) - \widehat{\nu_2}(h)|^2}{|h|^2} \right)^{1/2} . \]
Here $(3 \pi^2 \log 2)/8 \approx 2.5654$, and the claim of the proposition follows.
\end{proof}

\begin{proof}[Proof of Theorem \ref{rankrlattice2}] In this proof constants and implied constants depend only on the distributions of $\xi_1, \xi_2, \dots, \xi_r, I$, the value of the infimum in \eqref{diophantinecond}, $d$ and $r$. We start by estimating $|\widehat{\nu}(h)|$, $h \in \mathbb{Z}^d$. For every $1 \le i \le r$ let $D_i>0$ be the greatest common divisor of the (finite or infinite) set of integers
\[ \{ a-b : \Pr (\xi_i=a), \Pr (\xi_i=b)>0 \} . \]
That is, $\xi_i$ is a lattice variable with maximal span $D_i$. Note that these are well defined because $\xi_1, \xi_2, \dots, \xi_r$ are assumed to be nondegenerate. It is not difficult to see that the characteristic function $\varphi_i$ of $\xi_i$ satisfies
\begin{equation}\label{phiiestimate}
|\varphi_i (2 \pi x)| \le 1-c_i \| D_i x \|_{\mathbb{R}/\mathbb{Z}}^2
\end{equation}
for all $x \in \mathbb{R}$ with some constant $c_i>0$. Indeed, $|\varphi_i (2 \pi x)|=1$ if and only if $x$ is an integer multiple of $1/D_i$; in addition, an estimate of the form \eqref{phiiestimate} holds for all $x$ in an open neighborhood of $0$, see Petrov \cite[p.\ 14]{P}. Since both sides of \eqref{phiiestimate} have period $1/D_i$, \eqref{phiiestimate} follows with a small enough constant $c_i>0$.

Let $e(x)=e^{2 \pi i x}$, and let $D>0$ denote the least common multiple of the integers $D_i$, $1 \le i \le r$. By the construction of $\nu$ and \eqref{phiiestimate},
\begin{equation}\label{hatnuestimate}
\begin{split} |\widehat{\nu}(h)| &= \left| \sum_{i=1}^r \sum_{n \in \mathbb{Z}} \Pr (I=i) \Pr (\xi_i=n) e(-\langle h, n \alpha_i \rangle) \right| \\ &\le \sum_{i=1}^r \Pr (I=i) \left| \sum_{n \in \mathbb{Z}} \Pr (\xi_i=n) e(-\langle h, \alpha_i \rangle n) \right| \\ &= \sum_{i=1}^r \Pr (I=i) \left| \varphi_i (-2 \pi \langle h, \alpha_i \rangle ) \right| \\ &\le 1- \sum_{i=1}^r \Pr (I=i) c_i \| D_i \langle h, \alpha_i \rangle \|_{\mathbb{R}/\mathbb{Z}}^2 \\ &\le 1- \sum_{i=1}^r \Pr (I=i) c_i \frac{D_i}{D} \| D \langle h, \alpha_i \rangle \|_{\mathbb{R}/\mathbb{Z}}^2 \\ &\le 1-c \max_{1 \le i \le r} \| \langle Dh, \alpha_i \rangle \|_{\mathbb{R}/\mathbb{Z}}^2 \end{split}
\end{equation}
with some constant $c>0$. Note that we used the general inequality $\| D_i x \|_{\mathbb{R}/\mathbb{Z}} \ge (D_i/D) \| D x \|_{\mathbb{R}/\mathbb{Z}}$, $x \in \mathbb{R}$, which follows from the subadditivity of the function $\| \cdot \|_{\mathbb{R}/\mathbb{Z}}$, and in the last step the assumption $\Pr (I=i)>0$.

Let us apply Proposition \ref{ErdosTuran} to $\nu_1=\nu^{*k}$ and $\nu_2 = \mu$ with $H=\lfloor \psi^{-1} (\sqrt{k}) /(2D) \rfloor$; in particular, $\psi (2DH)^2 \le k$. By \eqref{hatnuestimate},
\[ \left| \widehat{\nu^{*k}}(h) -\widehat{\mu} (h) \right| = \left| \widehat{\nu}(h) \right|^k \le \exp \left( -ck \max_{1 \le i \le r} \| \langle Dh, \alpha_i \rangle \|_{\mathbb{R}/\mathbb{Z}}^2 \right) \]
for any $h \in \mathbb{Z}^d$, $h \neq 0$. It will thus be enough to prove
\begin{equation}\label{enough}
\sum_{\substack{h \in \mathbb{Z}^d \\ 0<\| h \|_{\infty}<DH}} \frac{\exp \left( -ck \max_{1 \le i \le r} \| \langle h, \alpha_i \rangle \|_{\mathbb{R}/\mathbb{Z}}^2 \right)}{|h|^2} \ll \frac{1}{H^2} .
\end{equation}

For any integer $1 \le m \le DH$, let
\begin{equation}\label{Bmdefinition}
B_m = \sum_{\substack{h \in \mathbb{Z}^d \\ 0<\| h \|_{\infty}<m}} \exp \left( -ck \max_{1 \le i \le r} \| \langle h, \alpha_i \rangle \|_{\mathbb{R}/\mathbb{Z}}^2 \right) .
\end{equation}
Let $C_0=[-K/(2 \psi (2m)), K/(2\psi(2m))]^r$ denote the axis parallel cube in $\mathbb{R}^r$ centered at the origin with edge length $K/\psi (2m)$, where the constant $K>0$ denotes the value of the infimum in \eqref{diophantinecond}. The translates $C_q:=C_0+q K/\psi (2m)$, $q \in \mathbb{Z}^r$ decompose $\mathbb{R}^r$ into congruent axis parallel cubes. Consider the point set
\[ P=\left\{ g (h) \,\, : \,\, h \in \mathbb{Z}^d, \, 0<\| h \|_{\infty} < m \right\}, \]
where $g(h)$ denotes the unique representative of $\left( \langle h, \alpha_1 \rangle, \langle h, \alpha_2 \rangle, \dots, \langle h, \alpha_r \rangle \right)+ \mathbb{Z}^r$ in $[-1/2, 1/2)^r$. Observe that $\| g(h)\|_{\infty}=\max_{1 \le i \le r} \| \langle h, \alpha_i \rangle \|_{\mathbb{R}/\mathbb{Z}}$. We claim that every cube $C_q$ contains at most one point of $P$. Indeed, let $h,h' \in \mathbb{Z}^d$ with $\| h \|_{\infty}, \| h' \|_{\infty} < m$, $h \neq h'$ be arbitrary. By the choice of $K$ and the strict monotonicity of $\psi$,
\[ \begin{split} \| g(h)-g(h') \|_{\infty} \ge \| g(h-h') \|_{\infty} &= \max_{1 \le i \le r} \| \langle h-h', \alpha_i \rangle \|_{\mathbb{R}/\mathbb{Z}} \\ &\ge K/\psi (\| h-h' \|_{\infty}) \\ &> K/\psi (2m) . \end{split} \]
Therefore $g(h)$ and $g(h')$ cannot lie in the same axis parallel cube with edge length $K/\psi (2m)$. Further, since $g(0)=0\in C_0$, no point of $P$ lies in $C_0$.

If $g(h)\in P \cap C_q$ for some $q \in \mathbb{Z}^r$, $q \neq 0$, then
\[ \| g(h) \|_{\infty} \ge \left( \| q \|_{\infty} - \frac{1}{2} \right) \frac{K}{\psi (2m)} \ge \| q \|_{\infty} \frac{K}{2 \psi (2m)} . \]
Therefore with the constant $a=cK^2/4 >0$,
\[ \begin{split} B_m = \sum_{g(h) \in P} \exp \left( -ck \| g(h) \|_{\infty}^2 \right) &\le \sum_{\substack{q \in \mathbb{Z}^r \\ q \neq 0}} \exp \left( -a k \| q \|_{\infty}^2 / \psi (2m)^2 \right) \\ &\ll \sum_{\ell =1}^{\infty} \ell^{r-1} \exp \left( -a k \ell^2 / \psi (2m)^2 \right) , \end{split} \]
where we used the fact that there are $\ll \ell^{r-1}$ lattice points $q \in \mathbb{Z}^r$ with $\| q \|_{\infty} = \ell$. Note that here the $\ell =1$ term dominates. Indeed, $k/\psi (2m)^2 \ge k/\psi (2DH)^2 \ge 1$ yields
\begin{equation}\label{Bmestimate}
\begin{split} B_m &\ll \exp \left( -ak/\psi (2m)^2 \right) \sum_{\ell =1}^{\infty} \ell^{r-1} \exp \left( -ak (\ell^2 -1)/\psi (2m)^2 \right) \\ &\le \exp \left( -ak/\psi (2m)^2 \right) \sum_{\ell =1}^{\infty} \ell^{r-1} \exp \left( -a (\ell^2 -1) \right) \\ &\ll \exp \left( -ak/\psi (2m)^2 \right) . \end{split}
\end{equation}

We prove \eqref{enough} using ``multidimensional summation by parts''. Formally, by the definition \eqref{Bmdefinition} of $B_m$, \eqref{Bmestimate} and $|h| \ge \| h \|_{\infty}$,
\begin{equation}\label{lastestimate}
\begin{split} \sum_{\substack{h \in \mathbb{Z}^d \\ 0<\| h \|_{\infty}<DH}} &\frac{\exp \left( -ck \max_{1 \le i \le r} \| \langle h, \alpha_i \rangle \|_{\mathbb{R}/\mathbb{Z}}^2 \right)}{|h|^2} \\ &\le \sum_{m=1}^{DH-1} \frac{1}{m^2} \sum_{\substack{h \in \mathbb{Z}^d \\ \| h \|_{\infty}=m}} \exp \left( -ck \max_{1 \le i \le r} \| \langle h, \alpha_i \rangle \|_{\mathbb{R}/\mathbb{Z}}^2 \right) \\ &= \sum_{m=1}^{DH-1} \frac{1}{m^2} (B_{m+1}-B_m) \\ &= \sum_{m=2}^{DH} \left( \frac{1}{m^2} - \frac{1}{(m+1)^2} \right) B_m + \frac{B_{DH}}{(DH+1)^2} \\ &\ll \sum_{m=2}^{DH} \frac{\exp \left( -ak/\psi (2m)^2 \right)}{m^3} + \frac{1}{H^2} . \end{split}
\end{equation}
Here
\[ \frac{k}{\psi (2m)^2} \ge \frac{\psi (2DH)^2}{\psi (2m)^2} \ge \left( \frac{DH}{m} \right)^{2d/r} . \]
Since the function $\exp \left( -a x^{2d/r} \right)x^3$ is bounded on $[1,\infty)$,
\[ \frac{\exp \left( -ak/\psi (2m)^2 \right)}{m^3} \le \frac{\exp \left( -a (DH/m)^{2d/r} \right)}{m^3} \ll \frac{1}{H^3} \]
for all $2 \le m \le DH$. Applying this estimate to each term in the last sum in \eqref{lastestimate}, the estimate \eqref{enough} follows.
\end{proof}

Finally, we prove Theorem \ref{rankrlattice3}. We start with two auxiliary lemmas. A finite set of points $A \subset \mathbb{R}^d / \mathbb{Z}^d$ is called an $R$-net, if the set of closed balls with center in $A$ and radius $R$ (in the Euclidean metric) cover all of $\mathbb{R}^d / \mathbb{Z}^d$.
\begin{lem}\label{wassersteinlowerbound} Let $\vartheta$ be a Borel probability measure on $\mathbb{R}^d / \mathbb{Z}^d$, and let $0<p \le 1$. If $A$ is a finite $R$-net of cardinality $|A|$, then
\[ W_p (\vartheta , \mu ) \ge \frac{d}{d+p} (\omega_d |A|)^{-p/d} - R^p (1-\vartheta (A)), \]
where $\omega_d$ is the volume of the unit ball in $\mathbb{R}^d$.
\end{lem}

\begin{proof} By Kantorovich duality, it will be enough to find a function $f: \mathbb{R}^d / \mathbb{Z}^d \to \mathbb{R}$ with $\| f \|_{p-\textrm{H\"old}} \le 1$, for which
\begin{equation}\label{wasslower}
\int_{\mathbb{R}^d / \mathbb{Z}^d} f \, \mathrm{d} \mu - \int_{\mathbb{R}^d / \mathbb{Z}^d} f \, \mathrm{d} \vartheta \ge \frac{d}{d+p} (\omega_d |A|)^{-p/d} - R^p (1-\vartheta (A)) .
\end{equation}
We claim that $f(x)=\mathrm{dist}(x,A)^p$ satisfies \eqref{wasslower}, where $\mathrm{dist}(x,A)$ denotes the Euclidean distance from the point $x$ to the set $A$. By a trivial estimate, for any $t \ge 0$ we have
\[ \mu \left( \left\{ x \in \mathbb{R}^d / \mathbb{Z}^d : \mathrm{dist} (x,A) \le t  \right\} \right) \le |A| \omega_d t^d, \]
and hence
\[ \begin{split} \int_{\mathbb{R}^d / \mathbb{Z}^d} \mathrm{dist}(x,A)^p \, \mathrm{d} \mu (x) &= \int_0^{\infty} \mu \left( \left\{ x \in \mathbb{R}^d / \mathbb{Z}^d : \mathrm{dist} (x,A)^p > t \right\} \right) \, \mathrm{d}t \\ &\ge \int_{0}^{(\omega_d |A|)^{-p/d}} \left( 1-|A| \omega_d t^{d/p} \right) \, \mathrm{d}t \\&= \frac{d}{d+p} (\omega_d |A|)^{-p/d} . \end{split} \]
On the other hand,
\[ \int_{\mathbb{R}^d / \mathbb{Z}^d} \mathrm{dist} (x,A)^p \, \mathrm{d} \vartheta (x) \le R^p (1-\vartheta (A)), \]
because the integrand $\mathrm{dist} (x,A)^p$ is zero on the set $A$, and at most $R^p$ everywhere by the fact that $A$ is an $R$-net. The last two relations together imply \eqref{wasslower}, and we are done.
\end{proof}

\begin{lem}\label{Rnetlemma} Assume that the system of vectors $\alpha_1, \alpha_2, \dots, \alpha_r \in \mathbb{R}^d$ is badly approximable, and let $J_1, J_2, \dots, J_r \subset \mathbb{Z}$ be finite intervals of integers of size $|J_i| \ge L \ge 1$ ($1 \le i \le r$). Then the set
\[ A=\left\{ \sum_{i=1}^r n_i \alpha_i + \mathbb{Z}^d : n_i \in J_i \textrm{ for all } 1 \le i \le r \right\} \subset \mathbb{R}^d / \mathbb{Z}^d \]
is an $R$-net for some $R \ll L^{-r/d}$, with an implied constant depending only on $\alpha_1, \alpha_2, \dots, \alpha_r$.
\end{lem}

\begin{proof} We may assume that $L \ge 3$ is an odd integer. According to a classical transference principle in the theory of Diophantine approximation \cite[p.\ 84]{CA}, a system of vectors $\alpha_1, \alpha_2, \dots, \alpha_r \in \mathbb{R}^d$ is badly approximable if and only if there exists a constant $c>0$ such that for any integer $N \ge 1$ and any point $x \in \mathbb{R}^d$ the system of inequalities
\[ \begin{split} \left\| \sum_{i=1}^r n_i \alpha_i -x \right\|_{\mathbb{R}^d/\mathbb{Z}^d} &\le c N^{-r/d}, \\ |n_1|, |n_2|, \dots, |n_r| &\le N \end{split} \]
has an integer solution $n_1, n_2, \dots, n_r$. In particular, the set
\[ \left\{ \sum_{i=1}^r n_i \alpha_i + \mathbb{Z}^d : n_i \in \mathbb{Z}, |n_i| \le \frac{L-1}{2} \textrm{ for all } 1 \le i \le r \right\} \]
is an $R$-net with $R=c((L-1)/2)^{-r/d} \ll L^{-r/d}$; being a translate of this set, so is $A$.
\end{proof}

\begin{proof}[Proof of Theorem \ref{rankrlattice3}] In this proof constants and implied constants depend only on $\alpha_1, \alpha_2, \dots, \alpha_r$ and $B$. Consider a sequence of i.i.d.\ $\mathbb{Z}^r$-valued random variables $V_1, V_2, \dots$, where $V_1=\xi_I e_I$; here $e_1, e_2, \dots, e_r$ denote the standard basis vectors in $\mathbb{Z}^r$. Observe that the $i$th coordinate of $V_1$ has expected value $E_i:=\Pr (I=i) \E \xi_i$, and variance
\[ \Pr (I=i) \E \xi_i^2 - \Pr (I=i)^2 (\E \xi_i )^2 \le \Pr (I=i) \E \xi_i^2 \le B. \]
Applying the Chebyshev inequality in each coordinate, for any $\lambda >0$,
\[ \Pr \left( \sum_{j=1}^k V_j \not\in \prod_{i=1}^r \left[ E_i k - \lambda \sqrt{k}, E_i k + \lambda \sqrt{k} \right]  \right) \le r \frac{B}{\lambda^2} \ll \frac{1}{\lambda^2}. \]
The random walk $\sum_{j=1}^k V_j$ is mapped to our random walk $S_k=\sum_{j=1}^k X_j$ by the function $g: \mathbb{Z}^r \to \mathbb{R}^d/\mathbb{Z}^d$, $g(n_1, n_2, \dots, n_r) = \sum_{i=1}^r n_i \alpha_i + \mathbb{Z}^d$. Therefore the set
\[ A:= \left\{ \sum_{i=1}^r n_i \alpha_i + \mathbb{Z}^d : n_i \in \left[ E_i k -\lambda \sqrt{k}, E_i k +\lambda \sqrt{k} \right] \cap \mathbb{Z} \textrm{ for all } 1 \le i \le r \right\} \]
satisfies $\Pr (S_k \not\in A) \ll 1/\lambda^{2}$ as well; in other words, $1-\nu^{*k}(A) \ll 1/\lambda^2$. On the other hand, by Lemma \ref{Rnetlemma} the set $A$ is an $R$-net with some $R \ll \lambda^{-r/d} k^{-r/(2d)}$ of cardinality $|A| \ll \lambda^{r}k^{r/2}$. Choosing $\lambda>0$ to be a large enough constant, Lemma \ref{wassersteinlowerbound} thus gives
\[ W_p (\nu^{*k}, \mu)^{1/p} \ge \left( \frac{d}{d+p} (\omega_d |A|)^{-p/d} - R^p (1-\nu^{*k} (A)) \right)^{1/p} \gg k^{-r/(2d)}, \]
as claimed.
\end{proof}

\subsection*{Acknowledgments}

The author is supported by the Austrian Science Fund (FWF), project Y-901.

\end{document}